\RequirePackage{rotating}
\documentclass[a4paper,10pt,reqno]{amsart}
\input{preamble.sty}
\usepackage[margin=1.5in]{geometry}
\usepackage[T1]{fontenc}
\usepackage[utf8]{inputenc}
\usepackage{palatino}
\usepackage{textcomp}
\usepackage{overpic}
\usepackage{pgf}
\usepackage{pgfkeys}
\usepackage{pgfplots}
\pgfplotsset{compat=1.18}
\usepgfplotslibrary{units}
\usepackage{graphicx}
\usepackage{thm-restate}

\graphicspath{{figures}}
\usepackage{svg}
\usepackage{float}

\usepackage{xcolor}
\definecolor{darkblue}{RGB}{0, 0, 166} 
\definecolor{green}{RGB}{0,102,0} 

\newcommand{\RR}{\mathbb{R}}      % for Real numbers
\newcommand{\ZZ}{\mathbb{Z}}      % for Integers

\newcommand{\QQ}{\mathbb{Q}}
\DeclareMathOperator{\Aut}{Aut}

\theoremstyle{definition}
\newtheorem{definition}[thm]{Definition}
\newtheorem{remark}[thm]{Remark}

\newtheorem{question}[thm]{Question}

\usepackage{hyperref}
\usepackage{xcolor}
\hypersetup{
	colorlinks=true,
	linkcolor={magenta},
	citecolor={blue},
	urlcolor={black}
}

\usepackage{cleveref}
\usepackage[backend=bibtex,maxbibnames=99,giveninits=true,style=ieee,sorting=nty,style=numeric]{biblatex}
\DeclareFieldFormat{postnote}{#1}
\allowdisplaybreaks

\title[Chain link surgeries and $\chi$-slice 3-braid closures]{On chain link surgeries bounding rational homology balls and $\chi$-slice 3-braid closures}
\author{Vitalijs Brejevs and Jonathan Simone}

\begin{document}

    \begin{abstract}
        We determine which integral surgeries on a large class of circular chain links bound rational homology balls. Our key tool is the lattice-theoretic cubiquity obstruction recently developed by Greene and Owens in \cite{greene-owens}. We discuss a practical method of computing it, and, as an application, prove that a generalisation of the slice--ribbon conjecture holds for all but one infinite family of quasi-alternating 3-braid links, which extends previous results of Lisca concerning the conjecture for 3-braid knots.
        \end{abstract}
\maketitle

\section{Introduction}
\label{sec:intro}

The question of which rational homology 3-spheres ($\Q S^3$s) bound rational homology 4-balls ($\Q B^4$s) is a well-known problem in low-dimensional topology \cite[Problem~4.5]{kirbylist}. A rich source of $\Q S^3$s is the double branched cover construction: if $K \subset S^3$ is a knot, then the double cover of $S^3$ branched along $K$, denoted $\Sigma_2(K)$, is a $\Q S^3$. Moreover, if $K$ is \emph{slice}, i.e., if $K$ bounds a properly smoothly embedded disc $D \subset B^4$, then $\Sigma_2(D)$, the double cover of $B^4$ branched along $D$, is a $\Q B^4$ bounded by $\Sigma_2(K)$. This statement generalises to links in the following way. Say that $S$ is a \emph{slice surface} for a link $L \subset S^3$ if $S$ is properly smoothly embedded in $B^4$, has no closed components, and $\pp S = L $; we do not require that $S$ be connected or orientable. Then we call $L$ a \emph{$\chi$-slice} link if $L$ admits a slice surface $S$ of Euler characteristic one. Donald and Owens have shown in~\cite{DonaldOwens2012} that if $L$ is $\chi$-slice and has non-zero determinant, then $\Sigma_2(S)$ is a $\Q B^4$ bounded by $\Sigma_2(L)$.

The present article explores the family of $\Q S^3$s that arise as double branched covers of 3-braid closures. We first describe the $\Q S^3$s in question as surgeries along \textit{chain links}, and then consider the $\chi$-sliceness of the underlying 3-braid links.

\subsection{Surgeries on twisted chain links} 
Consider the 3-manifolds given by the surgery diagram in Figure~\ref{fig:P-intro}. Such surgeries were studied at length by the second author in~\cite{simone2020classification} whence we recall some terminology and notation. Call the underlying $n$-component link a \textit{$t$-half twisted chain link} and denote it by $L_n^t$. Writing $\textbf{x}=(x_1,\ldots,x_n)$, where $x_i \in \ZZ$ for all $i$, we denote the corresponding surgery 3-manifold by $S^3_{\textbf{x}}(L_n^t)$. By $-\textbf{x}$ we mean the string $(-x_1,\ldots,-x_n)$. Note that if $\textbf{x}'$ is any cyclic reordering and/or reversal of $\textbf{x}$, then $S^3_{\textbf{x}}(L_n^t)$ and $S^3_{\textbf{x}'}(L_n^t)$ are diffeomorphic. In Section~\ref{sec:dbcs} we will show that any integral chain link surgery is diffeomorphic to a chain link surgery in one of three standard forms:

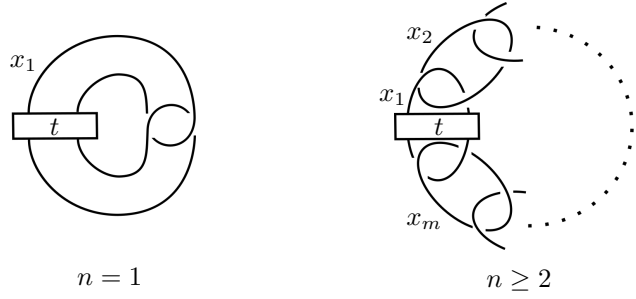
\begin{figure}[h]
\centering
\begin{overpic}
[scale=.6]{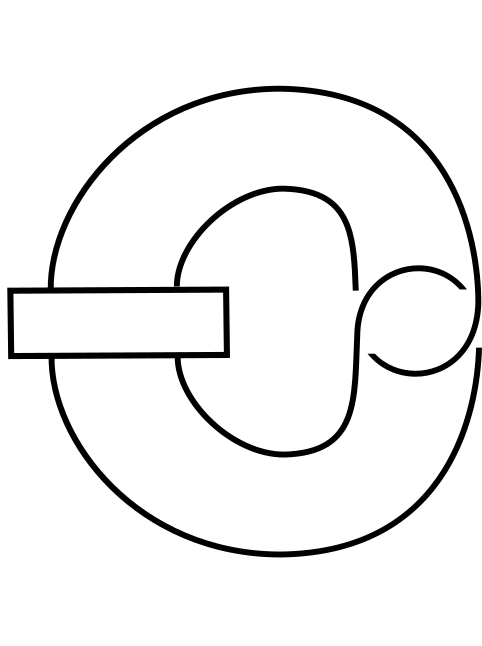}
\put(0,75){$x_1$}
\put(16,47){$t$}
\put(27,-14){$n=1$}
\end{overpic}\hspace{1in}
\begin{overpic}
[scale=.6]{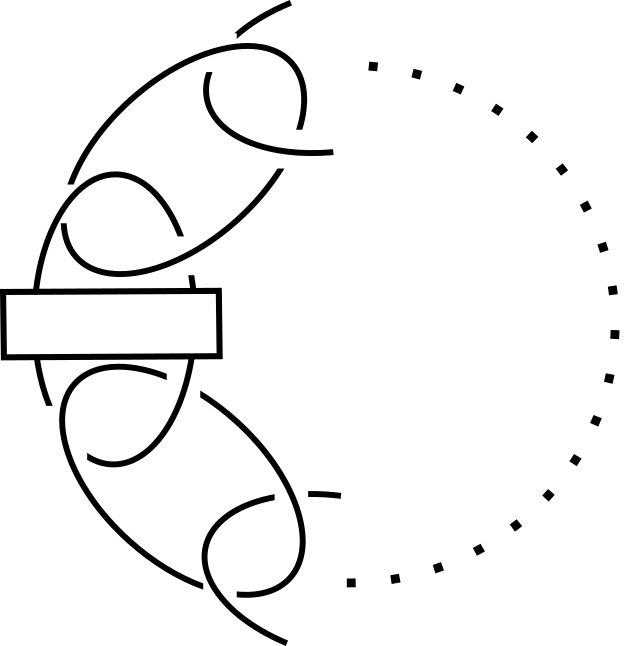}
\put(-6,60){$x_1$}
\put(5,85){$x_2$}
\put(5,10){$x_m$}
\put(16,46.5){$t$}
\put(37,-15){$n\ge2$}
\end{overpic}\vspace{.4cm}
\caption{Integral surgery along an $n$-component $t$-half twisted chain link $L_n^t$, which we denote by $S^3_{\textbf{x}}(L_n^t)$, where $\textbf{x}=(x_1,\ldots,x_n)$ and $x_i\in\mathbb{Z}$ for all $i$. The box labeled $t$ indicates the number of half-twists.}\label{fig:P-intro}
\end{figure}

\begin{prop} Let $\textbf{x}=(x_1,\ldots,x_{m})$. 
Then $S^3_{\textbf{x}}(L_m^s)$ 
is diffeomorphic to some $S^3_{\textbf{a}}(L_n^t)$, where $\textbf{a}=(a_1,\ldots,a_n)$ and either:
\begin{enumerate}[label=(\roman*)]
    \item $n=1$, $\textbf{a}=(a)$, and 
        \begin{enumerate}
            \item $a\in\{-1,-2,-3\}$ if $t$ is odd, or
            \item $a\in\{1,2,3\}$ if $t$ is even;
        \end{enumerate}
    \item $n=2$, $a_1=0$, and $a_2\in\mathbb{Z}$; or 
    \item 
    \begin{enumerate}
        \item $n=1$, $\textbf{a}=(a)$, and either $t$ is even and $a\le -1$, or $t$ is odd and $a\le -5$, or
        \item $n\ge 2$, $a_i\le -2$ for all $i$, and there exists $j$ such that $a_j\le -3$.
    \end{enumerate}
\end{enumerate}
\label{prop:surgerydiagram}
\end{prop}

It is shown in \cite{simone2020classification} that if $S^3_{\textbf{a}}(L_n^t)$ is of type (iii), then it is a $\Q S^3$. It is now easy to see from the surgery diagrams that $S^3_{\textbf{a}}(L_n^t)$ is not a $\Q S^3$ if and only if it is of type (ii) with $t$ even. 
The following result, to be proven in Section~\ref{sec:dbcs}, almost completely describes the $\QQ S^3$s of the first two types that bound $\QQ B^4$s, with the exception of an infinite family of Brieskorn spheres (see Remark~\ref{rem:brieskorn}).

\begin{prop}\hfill
\begin{enumerate}
\item Let $S^3_{(a)}(L_1^t)$ be of type (i) and suppose it is not the case that $a=1$ and $t\ge10$ is even, or $a=-1$ and $t\le-11$ is odd. Then $S^3_{(a)}(L_1^t)$ bounds a $\Q B^4$ if and only if $(t,a)\in\{(2n,1), (-2n-1,-1) \mid 0\le n\le 4\}$.
\item If $S^3_{\textbf{a}}(L_2^t)$ is of type (ii), then it bounds a $\QQ B^4$ if and only if $t$ is odd.
\end{enumerate}
\label{prop:parabolic}
\end{prop}

\begin{rem} Proposition \ref{prop:parabolic} provides a full classification of which integral surgeries on chain links belonging to families (i) and (ii) bound $\Q B^4s$ except for those in type~(i) with $a=1$ and $t\ge10$ even, and $a=-1$ and $t\le -11$ odd. These are precisely the Brieskorn spheres $\Sigma(2,3,6n+1)$, where $n\ge 5$ (cf.~the proof of Proposition \ref{prop:parabolic}). This family has been studied for decades, but it is still unknown for which values of $n\ge5$ the manifold $\Sigma(2,3,6n+1)$ bounds a $\Q B^4$. 
\label{rem:brieskorn}
\end{rem}

We now assume that $S^3_{\textbf{a}}(L_n^t)$ is of type (iii). In~\cite{simone2020classification}, this manifold is given simpler notation that unifies the $n=1$ and $n\ge2$ cases (cf.~Lemma \ref{lem:dbc}(3)). We adopt this notation here: 
\[ Y_\textbf{a}^t=\begin{cases} 
      S^3_{-\textbf{a}}(L_n^t) & \text{ if } n\ge2, \\
      S^3_{(-a_1+2)}(L_1^t) & \text{ if $n=1$ and $t$ is even,} \\ 
        S^3_{(-a_1-2)}(L_1^t) & \text{ if $n=1$ and $t$ is odd.} \\ 
   \end{cases}
\]

For $t\in\{-1,0,1\}$, the article~\cite{simone2020classification} provides an almost complete understanding of which strings $\textbf{a}$ yield $Y^{t}_{\textbf{a}}$ that bound $\Q B^4$s. This depends on whether $\textbf{a}$ belongs to particular explicitly defined sets, denoted by $\mathcal{S}_{kx}, \mathcal{S}_{kx}^*$, and $\mathcal{O}$ for $k \in \{ 1, 2 \}$ and $x \in \{ a, b, c, d, e\}$, with $\mathcal{S}_k = \bigcup_{x \in \{a,b,c,d,e\}} \mathcal{S}_{kx}$ and $\mathcal{S}_k^* = \bigcup_{x \in \{a,b,c,d,e\}} \mathcal{S}_{kx}^*$. We defer the precise definitions of these sets to Section~\ref{sec:notation}.

\begin{thm}[Theorem 1.7 in \cite{simone2020classification}] Let $\textbf{a}=(a_1,\ldots,a_n)$, where $a_i\ge2$ for all $i$ and $a_j\ge 3$ for some $j$.
\begin{enumerate} 
    \item $Y^0_{\textbf{a}}$ bounds a $\Q B^4$ if and only if $\textbf{a}\in\mathcal{S}_2\cup\mathcal{S}_2^*$.\label{(1)}
    \item If $\textbf{a}\not\in \mathcal{S}^*_{1a}\cup \mathcal{O}$, then $Y^{-1}_{\textbf{a}}$ bounds a $\Q B^4$ if and only if $\textbf{a}\in \mathcal{S}_1\cup (\mathcal{S}^*_1\setminus \mathcal{S}_{1a}^*)$.\label{(2)}
	\item If $\textbf{a}\not\in \mathcal{S}_{1a}\cup \mathcal{O}$, then $Y^{1}_{\textbf{a}}$ bounds a $\Q B^4$ if and only if $\textbf{a}\in (\mathcal{S}_1\setminus \mathcal{S}_{1a})\cup \mathcal{S}_1^*$.\label{(3)}
\end{enumerate}
\label{thm:s}\end{thm}

Notice that Theorem \ref{thm:s}(\ref{(1)}) provides a full classification of rational homology spheres of the form $Y^0_{\textbf{a}}$ that bound rational homology balls. One aim of this paper is to upgrade Theorem \ref{thm:s}(\ref{(2)}) and (\ref{(3)}) to obtain a (nearly) complete classification of rational homology spheres of the form $Y^{\pm1}_{\textbf{a}}$ that bound rational homology balls. In particular, we show the following. 

\begin{thm} 
Let $\mathbf{3}_6 = (3,3,3,3,3,3) \in \mathcal{O}$ and suppose $\textbf{a}\neq \mathbf{3}_6$.
\begin{enumerate}
\item $Y^{-1}_{\textbf{a}}$ bounds a $\Q B^4$ if and only if $\textbf{a}\in \mathcal{S}_1\cup (\mathcal{S}_{1}^*\setminus\mathcal{S}_{1a}^*)$.
\item $Y^{1}_{\textbf{a}}$ bounds a $\Q B^4$ if and only if $\textbf{a}\in (\mathcal{S}_1\setminus\mathcal{S}_{1a})\cup \mathcal{S}_{1}^*$.
\end{enumerate}
\label{thm:1}
\end{thm}

The proof of Theorem~\ref{thm:1} relies on an obstruction due to Greene and Jabuka~\cite{greene-pretzel} and developed by Greene and Owens~\cite{greene-owens}, called \textit{cubiquity}. It states that if a $\Q S^3$ bounds a $\Q B^4$ as well as a \textit{sharp} negative definite 4-manifold $X$, then the image of the embedding of the intersection lattice of $X$ into the standard integral lattice $\Z^n$ of equal rank, provided by Donaldson's diagonalization theorem, must intersect every unit cube of $\Z^n$. This additional geometric property follows from the consideration of Heegaard Floer homology $d$-invariants and their relationship to the lattice embedding. Further details will be provided in Section~\ref{sec:cubiquitydef}.

\subsection{The $\chi$-slice--ribbon conjecture and 3-braid closures}

We say that a link $L \subset S^3$ is \emph{$\chi$-ribbon} if it admits a slice surface $S$ of Euler characteristic one that can be smoothly isotoped rel boundary so that the radial distance function $B^4 \ra [0,1]$ induces a handle decomposition of $S$ with only 0- and 1-handles, in which case $S$ is called a \emph{ribbon surface} for $L$. This definition subsumes the usual notion of ribbonness for knots. The long-standing question of Fox~\cite{foxproblems} asking whether the sets of slice and ribbon knots coincide readily generalises to $\chi$-slice and $\chi$-ribbon links; we refer to this generalisation as the \emph{$\chi$-slice--ribbon conjecture}.

We will apply Theorem~\ref{thm:1} in order to prove the $\chi$-slice--ribbon conjecture for a large set of \emph{quasi-alternating (QA)} 3-braid links. To this end, we first recall the classification of 3-braids up to conjugacy due to Murasugi:

\begin{thm}[\cite{murasugi}]
\label{thm:murasugi}
Let $\sigma_1$ and $\sigma_2$ be the standard generators of the braid group on three strands $B_3$. Then any word in $B_3$ is equivalent, up to conjugation, to one of the following:
    \begin{enumerate}[label=(\roman*)]
        \item $(\sigma_1 \sigma_2)^{3t} \sigma_1^m \sigma_2^{-1}$, where  $m\in\{-1,-2,-3\}$;  \item $(\sigma_1 \sigma_2)^{3t} \sigma_2^{m}$, where $m\in\ZZ$; or
        \item $(\sigma_1 \sigma_2)^{3t} \sigma_1 \sigma_2^{-(a_1-2)} \cdots \sigma_1 \sigma_2^{-(a_n-2)}$, where $a_i \ge 2$ for all $i$, and $a_j \ge 3$ for some $j$.
    \end{enumerate}
\end{thm}

It follows that this is also a classification of links obtained as closures of 3-braids up to isotopy. For convenience, we introduce the following notation.

\begin{definition}
    We denote the closure of a 3-braid of: type (i) by $D_m^t$; type (ii) by $C_m^t$; and type (iii) by $B_\textbf{a}^t$, where $\textbf{a}=(a_1,\ldots,a_n)$.
\end{definition}

In light of Proposition \ref{prop:parabolic} and the relationship between 3-braids and integral chain link surgeries expounded in Section~\ref{sec:dbcs}, 
it will be a straightforward exercise to settle the $\chi$-sliceness of the 3-braid closures $C_m^t$ and $D_m^t$.

\begin{prop} $D_m^t$ is $\chi$-slice if and only if $(t,m)\in\{(0,-1),(1,-3)\}$. $C_m^t$ is $\chi$-slice for all $m,t$. Moreover, each of the $\chi$-slice links is indeed $\chi$-ribbon.
\label{prop:chi2,3}
\end{prop}

\begin{rem}
A 3-braid link has zero determinant if and only if it is of the form $C_m^t$ with $t$ is odd. We will see in Section \ref{sec:dbcs} that the double covers of $S^3$ branched along such links are precisely the chain link surgeries that are not (cf.~paragraph following the statement of Proposition~\ref{prop:surgerydiagram}). So even though the 3-braid closures themselves are $\chi$-slice, their double branched covers do not bound $\Q B^4s$.
\label{rem:notQS3}
\end{rem}

Generic 3-braid links come in form $B_\textbf{a}^t$. Indeed, we will see in Section~\ref{sec:dbcs} that $Y^t_{\textbf{a}}$ is the double cover branched along $B^t_{\textbf{a}}$. Note that it follows from Theorem~4.1 in~\cite{baldwinquasi} that the manifold $Y^t_{\textbf{a}}$ is an $L$-space and $B^t_{\textbf{a}}$ is QA if and only if $t\in\{-1,0,1\}$. 

In~\cite{brejevs2020ribbon}, the first author constructed Euler characteristic one ribbon surfaces for all links $B^0_{\textbf{a}}$ with $\textbf{a}\in \mathcal{S}_2 \setminus \mathcal{S}_{2c}$.
As a consequence of Theorem \ref{thm:1}, we can extend this result and say precisely which QA 3-braid links $B^t_\textbf{a}$ with $t=\pm1$ are $\chi$-slice.

\begin{thm}
Let $B=B^t_{\textbf{a}}$ be a QA 3-braid closure.  
\begin{enumerate}
    \item If $t=0$ and $\textbf{a}\not\in\mathcal{S}_{2c}$, then $B$ is $\chi$-slice if and only if and $\textbf{a}\in(\mathcal{S}_2\cup \mathcal{S}_2^*)\setminus\mathcal{S}_{2c}$. 
    \item If $t=-1$, then $B$ is $\chi$-slice if and only if $\textbf{a}\in \mathcal{S}_1\cup (\mathcal{S}_{1}^*\setminus\mathcal{S}_{1a}^*)$. 
    \item If $t=1$, then $B$ is $\chi$-slice if and only if $\textbf{a}\in (\mathcal{S}_1\setminus\mathcal{S}_{1a})\cup \mathcal{S}_{1}^*$.
\end{enumerate}
Moreover, every such $\chi$-slice link is $\chi$-ribbon.
\label{thm:2}
\end{thm}

\begin{remark}
\label{rem:S2c} It is not known, in general, which 3-braid closures in $\{B^0_{\textbf{a}}\}_{\textbf{a}\in\mathcal{S}_{2c}}$ are $\chi$-slice;
this family is the most mysterious. Although the double branched covers of such links all bound $\Q B^4$s by Theorem~\ref{thm:s}, this set contains links that are $\chi$-slice and links that are not $\chi$-slice. In particular, the article~\cite{brejevs2020ribbon} exhibits infinitely many strings $\textbf{a}\in\mathcal{S}_{2c}$ such that $B^0_{\textbf{a}}$ is $\chi$-slice; these are of the form $(3+m,3,3,2^{[m]},3,3)$. However, there exist strings $\textbf{a}\in\mathcal{S}_{2c}$ such that $B^0_{\textbf{a}}$ is not a slice knot; in particular, if $\textbf{a}= \mathbf{3}_i$ for $i \in \{ 7, 11, 17, 23 \}$, then $B^0_{\textbf{a}}$ is not slice by~\cite{ammmps2020branched,sartori2010}. Moreover, three more 3-braid knots in $\{B^0_{\textbf{a}}\}_{\textbf{a}\in\mathcal{S}_{2c}}$ are shown to be non-slice in~\cite{brejevs2020ribbon}; in particular, if $\textbf{a}=(2,4,2,4,4,2,4,2,3)$, $\textbf{a}=(2,2,4,3,2,5,2,3,4)$, or $\textbf{a}=(2,3,4,3,4,3,2,3,3)$, then $B^0_\textbf{a}$ is not $\chi-$slice. It is rather challenging to obstruct sliceness of these seven examples and requires an involved verification of the Herald--Kirk--Livingston condition~\cite{heraldkirklivingston} on their twisted Alexander polynomials. It is not known if there are infinitely many non-$\chi$-slice links in $\{B^0_{\textbf{a}}\}_{\textbf{a}\in\mathcal{S}_{2c}}$.
\end{remark}

It follows from the work of Lisca~\cite{lisca-3braids} that the slice--ribbon conjecture holds for all 3-braid knots with $\textbf{a}\not\in \mathcal{S}_{2c}$. Specifically, he showed that finite concordance order 3-braid knots are QA and belong to one of three infinite families, two of which are comprised of ribbon knots, whilst the third family is precisely $\{B^0_{\textbf{a}}\}_{\textbf{a}\in\mathcal{S}_{2c}}$. Hence, Theorem~\ref{thm:2} yields an extension of this result to QA 3-braid links:

\begin{thm} The $\chi$-slice--ribbon conjecture holds for all QA 3-braid links not in $\{B_\textbf{a}^0\}_{\textbf{a}\in \mathcal{S}_{2c}}$.
\label{cor:chisliceribbon}
\end{thm}

\subsection{Summary of results and questions} For easy reference, we will quickly summarize what precisely is known about the following questions:
\begin{itemize}
\item Which chain link surgeries $S^3_\textbf{a}(L_n^t)$ bound $\Q B^4$s?
\item Which 3-braid closures are $\chi$-slice?
\end{itemize}

\noindent By Proposition \ref{prop:surgerydiagram} and Theorem \ref{thm:murasugi}, the sets of chain link surgeries $S^3_\textbf{a}(L_n^t)$ and 3-braid closures can each be partitioned into three subsets.
Moreover, these subsets are related by the double branched cover construction. In particular, we will see in Section~\ref{sec:dbcs} that:
\begin{enumerate}[label=(\roman*)]
    \item $\Sigma_2(D^{t}_{m})=S^3_{m+4}(L_1^{t-1})$ if $t$ is odd and $\Sigma_2(D^{t}_{m})=S^3_{m}(L_1^{t-1})$ if $t$ is even;
    \item $\Sigma_2(C^t_{m})=S^3_{(0,m)}(L_2^{t-1})$;
    \item $\Sigma_2(B_{\textbf{a}}^t)=Y^t_{\textbf{a}}$. 
\end{enumerate}
We first consider case (ii), which is completely resolved.

\begin{center}
\begin{tabular}{|c|c|}
\hline
   \multicolumn{2}{|c|}{ \begin{tabular}{@{}l@{}} Let $C_m^t$ be the closure of the 3-braid $(\sigma_1 \sigma_2)^{3t} \sigma_2^{m}$, where $m\in\ZZ$\end{tabular}}\\\hline\hline
    \begin{tabular}{@{}c@{}}$\Sigma_2(C_m^t)=S^3_{(0,m)}(L_2^{t-1})$ bounds a $\QQ B^4$\\ if and only if $t-1$ is odd\end{tabular}& $C_m^t$ is $\chi$-slice for all $t$ and $m$\\\hline
\end{tabular}
\end{center}
\vspace{.1cm}

\noindent Next we have case (i), which is not completely resolved for chain link surgeries, but is completely resolved for 3-braid closures. Completely resolving this case for chain link surgeries would require one to understand which Brieskorn spheres $\Sigma(2,3,6n+1)$ bound $\Q B^4$s  (cf.~Remark \ref{rem:brieskorn}).

\begin{center}
\begin{tabular}{|c|c|c|}
\hline
    \multicolumn{3}{|c|}{\begin{tabular}{@{}l@{}} Let $D_m^t$ be the closure of the 3-braid $(\sigma_1 \sigma_2)^{3t} \sigma_1^m \sigma_2^{-1}$, where  $m\in\{-1,-2,-3\}$\end{tabular}}\\\hline\hline
    $t$ odd &
    \begin{tabular}{@{}c@{}} 
    (assuming $m\neq -3$ or $t\le 10$)\\
    $\Sigma_2(D_m^t)=S^3_{(m+4)}(L_1^{t-1})$ bounds a $\QQ B^4$\\
    if and only if 
    $(t,m)\in \{(2n,-3)\,|\,0\le n\le4\}$\end{tabular}
    & 
    \begin{tabular}{@{}c@{}}
    $C_m^t$ is $\chi$-slice
    if and only if\\
    $(t,m)=(1,-3)$\end{tabular}
    \\\hline
     $t$ even &
    \begin{tabular}{@{}c@{}} 
    (assuming $m\neq-1$ or $t\ge -10$)\\
    $\Sigma_2(D_m^t)=S^3_{(m)}(L_1^{t-1})$ bounds a $\QQ B^4$\\
    if and only if 
    $(t,m)\in \{-2n,-1)\,|\,0\le n\le 4\}$\end{tabular}
    & 
    \begin{tabular}{@{}c@{}}
    $C_m^t$ is $\chi$-slice
    if and only if\\
    $(t,m)=(0,-1)$\end{tabular}
    \\\hline
\end{tabular}
\end{center}
\vspace{.15cm}

\noindent We now consider case (iii), which constitutes the bulk of the examples. This case is the furthest from being fully resolved. We first consider the case in which $Y_\textbf{a}^t$ is an $L$-space, or equivalently, the case in which $B_\textbf{a}^t$ is QA; this occurs when $t\in\{-1,0,1\}$ (\cite{baldwinquasi}).
The main results of this paper provide a complete classification $\chi-$slice links of the form $B_\textbf{a}^{\pm1}$, and an almost complete classification of chain link surgeries of the form $Y_\textbf{a}^{\pm1}$ that bound $\QQ B^4$s (with the exception of $\textbf{a}=\textbf{3}_6$). The case of $t=0$ is completely resolved for chain link surgeries and mostly resolved for 3-braid closures, except for the mysterious family stemming from the set $\mathcal{S}_{2c}$ (see Remark~\ref{rem:S2c}).

\begin{center}
\begin{tabular}{|c|c|c|}
    \hline
    \multicolumn{3}{|c|}{ 
    \begin{tabular}{@{}c@{}} Let $B_\textbf{a}^t$ be the closure of the 3-braid $(\sigma_1 \sigma_2)^{3t} \sigma_1 \sigma_2^{-(a_1-2)} \cdots \sigma_1 \sigma_2^{-(a_n-2)}$,\\ where $a_i \ge 2$ for all $i$, and $a_j \ge 3$ for some $j$\end{tabular}}\\\hline\hline
    $t$ & 
    $Y_\textbf{a}^t$ bounds a $\QQ B^4$ & $B_\textbf{a}^t$ is
    \\\hline
    
    $0$ &
    if and only if $\textbf{a}\in\mathcal{S}_2\cup\mathcal{S}_2^*$ &
    \begin{tabular}{@{}c@{}}
    $\chi$-slice if $\textbf{a}\in(\mathcal{S}_2\cup\mathcal{S}_2^*)\setminus\mathcal{S}_{2c}$\\ 
    or $\textbf{a}=(3+m,3,3,2^{[m]},3,3)\in\mathcal{S}_{2c}$\\
    \hline 
    not $\chi$-slice if $\textbf{a}$ one of the following strings in $\mathcal{S}_{2c}$:\\
    $\textbf{3}_i\in\mathcal{S}_{2c}$ for $i\in\{7,11,17,23\}$,\\
    $(2,4,2,4,4,2,4,2,3)$, $(2,2,4,3,2,5,2,3,4)$,\\ or $(2,3,4,3,4,3,2,3,3)$
    \end{tabular}\\\hline
    
    $-1$  &
    \begin{tabular}{@{}c@{}} (assuming $\textbf{a}\neq \textbf{3}_6$) \\ if and only if $\textbf{a}\in \mathcal{S}_1\cup (\mathcal{S}^*_1\setminus \mathcal{S}_{1a}^*)$\end{tabular}  &
    $\chi$-slice
    if and only if $\textbf{a}\in \mathcal{S}_1\cup (\mathcal{S}_{1}^*\setminus\mathcal{S}_{1a}^*)$\\\hline 
    
    $1$ & \begin{tabular}{@{}c@{}} (assuming $\textbf{a}\neq\textbf{3}_6$)\\ if and only if $\textbf{a}\in (\mathcal{S}_1\setminus \mathcal{S}_{1a})\cup \mathcal{S}_1^*$\end{tabular}&
    $\chi$-slice
    if and only if $\textbf{a}\in (\mathcal{S}_1\setminus\mathcal{S}_{1a})\cup \mathcal{S}_{1}^*$\\\hline
\end{tabular}
\end{center}
\vspace{.1cm}

\begin{question}
Does $Y^{\pm1}_{\mathbf{3}_6}$ bound a $\Q B^4$? 
\label{q:1}
\end{question}

\begin{question}
For which $\textbf{a}\in\mathcal{S}_{2c}$ is $B_\textbf{a}^t$ $\chi$-slice? 
\label{q:2}
\end{question}

\noindent For non-QA 3-braid closures, much less is known. In particular, it is not known if any of non-QA 3-braid closures are $\chi$-slice. We also know little about which of the chain link surgeries bound $\Q B^4$s. The following comes as a corollary of the proof of Theorem~\ref{thm:s}.

\begin{center}
\begin{tabular}{|c|c|c|}
\hline
\multicolumn{3}{|c|}{ 
    \begin{tabular}{@{}c@{}} Let $B_\textbf{a}^t$ be the closure of the 3-braid $(\sigma_1 \sigma_2)^{3t} \sigma_1 \sigma_2^{-(a_1-2)} \cdots \sigma_1 \sigma_2^{-(a_n-2)}$,\\ where $a_i \ge 2$ for all $i$, and $a_j \ge 3$ for some $j$\end{tabular}}
    \\\hline\hline
    $t$ & $Y_\textbf{a}^t$ & $B_\textbf{a}^t$ 
    \\\hline
    
    even & \begin{tabular}{@{}c@{}} does not bound a $\Q B^4$\\ if $\textbf{a}\not\in\mathcal{S}_2\cup\mathcal{S}_2^*$\\\hline
    bounds a $\Q B^4$ if $\textbf{a}\in\mathcal{S}_{2c}$\end{tabular}& 
    is not $\chi$-slice if $\textbf{a}\not\in\mathcal{S}_2\cup\mathcal{S}_2^*$
    \\\hline
    
    odd  & \begin{tabular}{@{}c@{}} does not bound a $\Q B^4$\\ if $\textbf{a}\not\in \mathcal{S}_{1}\cup \mathcal{S}_{1}^*\cup \mathcal{O}$\end{tabular}&
    is not $\chi$-slice if $\textbf{a}\not\in\mathcal{S}_{1}\cup \mathcal{S}_{1}^*\cup \mathcal{O}$
    \\\hline
\end{tabular}
\end{center}
\vspace{.05cm}

\noindent The main obstacle here is that the obstructions used for QA links all vanish for the nonQA links not covered in the table above.

\begin{question} Does there exist a $\chi$-slice non-QA 3-braid link $B_\textbf{a}^t$? Does there exist $\textbf{a}\in \mathcal{S}_1\cup (\mathcal{S}_2\setminus\mathcal{S}_{2c})\cup\{\mathbf{3}_6\}$ and $t\not\in\{-1,0,1\}$ such that $Y^{t}_{\textbf{a}}$ bounds a $\Q B^4$?
\label{q:3}
\end{question}

\subsection{Organisation of the Paper}
In Section~\ref{sec:dbcs}, we show that 3-manifolds of the form $S^3_{\textbf{a}}(L_n^t)$ are precisely the double branched covers of 3-braid closures and use this to prove Propositions~\ref{prop:surgerydiagram},~\ref{prop:parabolic}, and~\ref{prop:chi2,3}. In Section~\ref{sec:notation} we define the sets $\mathcal{S}_i^*$, $\mathcal{S}_i$, and $\mathcal{O}$ that are used in the statements of the main theorems. In Section~\ref{sec:cubiquitydef} we introduce the cubiquity obstruction from~\cite{greene-pretzel} and~\cite{greene-owens}, as well as a practical method of computing it. The goal of Section~\ref{sec:sharpdbc} is to show that particular negative-definite 4-manifolds bounded by the chain link surgeries $Y^t_{\textbf{a}}$ are sharp whenever $t\le0$. Section~\ref{sec:latticebasic} contains the definitions of standard and circular subsets of $\ZZ^n$, and a condition under which such subsets are not cubiquitous (Theorem \ref{thm:Wodd}). The proof of Theorem~\ref{thm:1} follows in Section~\ref{sec:proof1}. Section~\ref{sec:ribbons} contains constructions of the ribbon surfaces claimed to exist in Proposition~\ref{prop:chi2,3} and Theorem~\ref{thm:2}. 

\subsection*{Acknowledgements} We would like to thank Brendan Owens for many helpful discussions, feedback on an earlier draft of this paper, and explanations of the algorithm for verifying cubiquity in Section~\ref{sec:cubiquitydef}. We also thank Frank Swenton for developing and maintaining \textsc{KLO} software.
Finally, we thank the anonymous  referee for providing detailed feedback and helpful comments. he first author was supported by the \emph{Fonds zur F\"{o}rderung der wissenschaftlichen Forschung} grant ``Cut and Paste Methods in Low Dimensional Topology''.

\subsection*{Data availability statement} The \textsc{SageMath} notebook used for the proof of Theorem~\ref{thm:1} is available on the first author's website: \url{https://vbrej.xyz/research}.

\section{Double Branched Covers of 3-Braid Closures}
\label{sec:dbcs}

In this section we will prove Propositions~\ref{prop:surgerydiagram}, \ref{prop:parabolic}, and~\ref{prop:chi2,3}. Their proofs rely on the relationship between chain link surgeries and 3-braid closures.

\begin{lem} For any string of integers $\textbf{x}=(x_1,\ldots,x_n)$, the manifold $S^3_{\textbf{x}}(L_n^t)$ is the double cover branched along the closure of the 3-braid given by one of:
\begin{itemize}
    \item $(\sigma_1\sigma_2)^{3t}\sigma_1\sigma_2^{x_1}$ if $n=1$ and $t$ is even; 
    \item $(\sigma_1\sigma_2)^{3(t+1)}\sigma_1^{-1}\sigma_2^{x_1}$ if $n=1$ and $t$ is odd; or 
    \item $(\sigma_1\sigma_2)^{3t}\sigma_1\sigma_2^{x_1+2}\sigma_1\sigma_2^{x_2+2}\cdots\sigma_1\sigma_2^{x_n+2}$ if $n\ge2$.
\end{itemize}

\label{lem:YisDBC}
\end{lem}

\begin{proof} 
Let $\textbf{x}=(x_1,\ldots,x_n)$, where $x_i\in\mathbb{Z}$ for all $i$. Consider the surfaces in Figure \ref{fig:dbc}, which are built from a single 0-handle and $n$ 1-handles; each labelled box indicates the number of half-twists. First assume $n\ge2$.
Following~\cite{akbulutkirby}, we get that $S^3_{\textbf{x}}(L_n^0)$ (resp., $S^3_{\textbf{x}}(L_n^{-1})$) is the double cover branched over the link given by the boundary of the surface shown on the top left (resp., top right) of Figure~\ref{fig:dbc}. The reader can verify that these links are isotopic to the closures of the 3-braids given by
\[
\sigma_1\sigma_2^{x_1+2}\sigma_1\sigma_2^{x_2+2}\cdots\sigma_1\sigma_2^{x_n+2} \quad \textrm{and} \quad (\sigma_1\sigma_2)^{-3}\sigma_1\sigma_2^{x_1+2}\sigma_1\sigma_2^{x_2+2}\cdots\sigma_1\sigma_2^{x_n+2},
\]
respectively.
Now, it follows from Dehn surgery arguments in Section~1.1 in~\cite{simone2020classification} that for any integer $t$, the manifold $S^3_{\textbf{x}}(L_n^t)$ is the double cover of $S^3$ branched along the closure of the 3-braid
\[
(\sigma_1\sigma_2)^{3t}\sigma_1\sigma_2^{x_1+2}\sigma_1\sigma_2^{x_2+2}\cdots\sigma_1\sigma_2^{x_n+2}.
\]

Next, let $n=1$ and set $\textbf{x}=(x)$. Similarly, following~\cite{akbulutkirby}, it can be shown that $S^3_{(x)}(L_1^{0})$ (resp., $S^3_{(x)}(L_1^{-1})$) is the double cover branched over the link given by the boundary of the surface shown on the bottom left (resp., bottom right) of Figure~\ref{fig:dbc}. Although these links are isotopic, it is useful to think of them as separate cases that are isotopic to the closures of the 3-braids $\sigma_1\sigma_2^{x}$ and $\sigma_1^{-1}\sigma_2^{x}$, respectively. Once again, following as in Section 1.1 in \cite{simone2020classification}, it can be shown via Dehn surgery that if $t$ is even (resp., $t$ is odd), then $S^3_{(x)}(L_1^t)$ is the double cover of $S^3$ branched along the closure of the 3-braid given by $(\sigma_1\sigma_2)^{3t}\sigma_1\sigma_2^{x}$ (resp., ($\sigma_1\sigma_2)^{3(t+1)}\sigma_1^{-1}\sigma_2^{x}$).  
\end{proof}

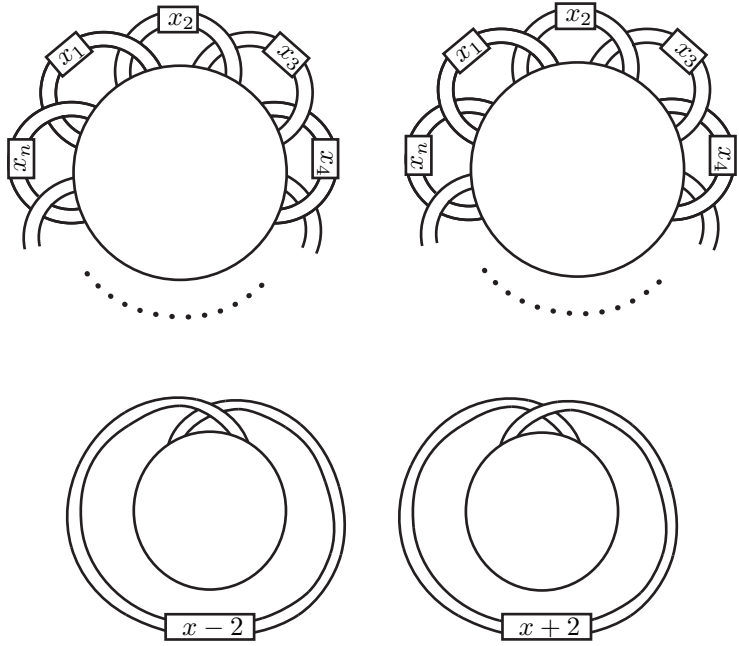
\begin{figure}[t]
\centering
\begin{overpic}
[scale=1]{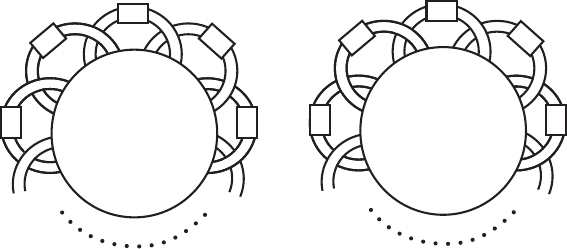}
\put(22,40.75){$x_2$}
\put(36.5,37){\rotatebox{-40}{$x_3$}}
\put(42.5,23.5){\rotatebox{270}{$x_4$}}
\put(6.75,35){\rotatebox{40}{$x_1$}}
\put(1,20.4){\rotatebox{90}{$x_{n}$}}

\put(76.5,41.3){$x_2$}
\put(91,37.3){\rotatebox{-40}{$x_3$}}
\put(97,24){\rotatebox{270}{$x_4$}}
\put(61.2,35.5){\rotatebox{40}{$x_1$}}
\put(55.4,20.7){\rotatebox{90}{$x_{n}$}}
\end{overpic}\vspace{1cm}

\begin{overpic}
[scale=1]{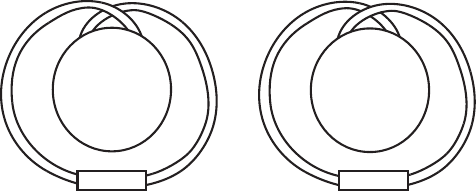}
\put(19,.9){$x-2$}
\put(74,.9){$x+2$}
\end{overpic}\vspace{.4cm}
\caption{Links whose double branched covers yield $Y^t_{\textbf{x}}$, $|t|\le1$.}\label{fig:dbc}
\end{figure}

\noindent Recall that: $B^t_{\textbf{a}}$ denotes the closure of the 3-braid
\[
(\sigma_1 \sigma_2)^{3t} \sigma_1 \sigma_2^{-(a_1-2)} \cdots \sigma_1 \sigma_2^{-(a_n-2)},
\]
where $\textbf{a}=(a_1,\ldots,a_n)$, $a_i \ge 2$ for $i=1,\dots,n$, and some $a_j \ge 3$; $C^t_{m}$ denotes the closure of $(\sigma_1 \sigma_2)^{3t} \sigma_2^{m}$; and $D^t_{m}$ denotes the closure of $(\sigma_1 \sigma_2)^{3t} \sigma_1^m \sigma_2^{-1}$, where $m\in\{-1,-2,-3\}$.

\begin{lem}\hfill
\begin{enumerate}
    \item Let $m\in\{-1,-2,-3\}$. Then $\Sigma_2(D^{t}_{m})=S^3_{m+4}(L_1^{t-1})$ if $t$ is odd and $\Sigma_2(D^{t}_{m})=S^3_{m}(L_1^{t-1})$ if $t$ is even.
    \item Let $m\in\mathbb{Z}$. Then $\Sigma_2(C^t_{m})=S^3_{(0,m)}(L_2^{t-1})$.
    \item Let $\textbf{a}=(a_1,\ldots,a_n)$, where $a_i\ge 2$ for all $i$ and $a_j\ge 3$ for some $j$. Then $\Sigma_2(B_{\textbf{a}}^t)=Y^t_{\textbf{a}}$. 
\end{enumerate}
\label{lem:dbc}
\end{lem}

\begin{proof}
    \par
    (1): Suppose $t$ is odd. By Lemma~\ref{lem:YisDBC}, $S^3_{(m+4)}(L_1^{t-1})$ is the double cover of $S^3$ branched along the closure of $(\sigma_1\sigma_2)^{3(t-1)}\sigma_1\sigma_2^{m+4}$. Write $\Delta = \sigma_1 \sigma_2 \sigma_1 = \sigma_2 \sigma_1 \sigma_2$ and recall that $\Delta^2$ is central in $B_3$ with $\Delta \sigma_1 = \sigma_2 \Delta$ and $\Delta \sigma_2 = \sigma_1 \Delta$. Writing $\sim$ to denote equivalence up to conjugation in $B_3$, we see that
        \begin{align*}
            (\sigma_1\sigma_2)^{3t}\sigma_1^{m} \sigma_2^{-1} &= \Delta^{2(t-1)} \Delta \sigma_2^m \sigma_1^{-1} \Delta = \Delta^{2(t-1)} (\sigma_2 \sigma_1 \sigma_2) \sigma_2^m \sigma_1^{-1} (\sigma_1 \sigma_2 \sigma_1) \\
            &= \Delta^{2(t-1)} \sigma_2 \sigma_1 \sigma_2^{m+2} \sigma_1 \sim \Delta^{2(t-1)} \sigma_1 \sigma_2 \sigma_1 \sigma_2^{m+2}\\
            &= \Delta^{2(t-1)} \sigma_2 \sigma_1 \sigma_2^{m+3} \sim \Delta^{2(t-1)} \sigma_1 \sigma_2^{m+4}\\
            &= (\sigma_1\sigma_2)^{3(t-1)}\sigma_1\sigma_2^{m+4}.
        \end{align*}
    Hence, $\Sigma_2(D^{t}_{m})=S^3_{(m+4)}(L_1^{t-1})$.

    Similarly, if $t$ is even, then Lemma \ref{lem:YisDBC} implies that $S^3_{(m)}(L_1^{t-1})$ is the double cover of $S^3$ branched along the closure of $(\sigma_1\sigma_2)^{3t}\sigma_1^{-1} \sigma_2^{m}$. We have
        \begin{align*}
            (\sigma_1\sigma_2)^{3t}\sigma_1^{m} \sigma_2^{-1} &= \Delta^{2t - 1} \sigma_2^m \sigma_1^{-1} \Delta \sim \Delta^{2t} \sigma_1^{-1} \sigma_2^m = (\sigma_1 \sigma_2)^{3t} \sigma_1^{-1} \sigma_2^m,
        \end{align*}
    so $\Sigma_2(D^{t}_{m})=S^3_{(m)}(L_1^{t-1})$.

    (2): Let $m \in \ZZ$. By Lemma \ref{lem:YisDBC}, $S^3_{(0,m)}(L_2^{t-1})$ is the double cover of $S^3$ branched along the closure of $(\sigma_1\sigma_2)^{3(t-1)}\sigma_1\sigma_2^2\sigma_1\sigma_2^{m+2}$. The statement follows since
        \begin{align*}
            (\sigma_1 \sigma_2)^{3t} \sigma_2^m &= \Delta^{2(t-1)} \sigma_1 \sigma_2 \sigma_1 \sigma_2 \sigma_1 \sigma_2^{m+1}\\
            &= \Delta^{2(t-1)} \sigma_1 \sigma_2^2 \sigma_1 \sigma_2^{m+2} = (\sigma_1\sigma_2)^{3(t-1)}\sigma_1\sigma_2^2\sigma_1\sigma_2^{m+2}.
        \end{align*}

    (3): Follows directly from Lemma~\ref{lem:YisDBC} and the definition of $Y^t_{\textbf{a}}$.\qedhere   
\end{proof}

\begin{proof}[Proof of Proposition~\ref{prop:surgerydiagram}]
Follows from Theorem~\ref{thm:murasugi} and Lemmas~\ref{lem:YisDBC} and~\ref{lem:dbc}.
\end{proof}

\begin{proof}[Proof of Proposition~\ref{prop:parabolic}]
Let $S^3_{\textbf{a}}(L_2^t)$ be of type (ii), where $\textbf{a}=(a,0)$.
If $t$ is even, then a quick homology calculation shows that $S^3_{\textbf{a}}(L_2^t)$ is not a $\Q S^3$, hence it cannot bound a $\Q B^4$. If $t$ is odd, then there is an obvious $\Q B^4$ bounded by $S^3_{\textbf{a}}(L_2^t)$ obtained by changing the 0-framed unknot to dotted circle notation; see Figure \ref{fig:QB}.

\begin{figure}
\centering
\begin{overpic}
    [scale=.7]{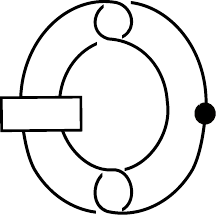}
    \put(16,42.5){$t$}
    \put(5,80){$a$}
\end{overpic}
\caption{A $\Q B^4$ bounded by $S^3_{(a,0)}(L_2^t)$.}\label{fig:QB}
\end{figure}

Now suppose $S^3_{(a)}(L_1^t)$ is of type (i). Since the order of the first homology of a $\Q S^3$ bounding a $\Q B^4$ must be a square (by, e.g., Lemma~3 in~\cite{cassongordon}), and $ 0 < |a| < 4$, it follows that $|a|=1$. If $t=0$, then $a=1$ and $S^3_{(a)}(L_1^t) = S^3$, which bounds $B^4$. 
Now, it is easy to see via surgery 
that if $t<0$, then reversing the orientation of $S^3_{(a)}(L_1^t)$ yields $S^3_{(-a)}(L_1^{-t+1})$ (cf.~Section~2.2 in~\cite{simone2020classification}). Thus we need only consider $S^3_{(a)}(L_1^t)$, where $t\ge0$.
Suppose $t$ is odd. Then by Lemma \ref{lem:dbc}, $S^3_{(-1)}(L_1^t)$ is the double cover of $S^3$ branched along the closure of
\[
    (\sigma_1\sigma_2)^{3(t+1)}\sigma_1^{-1}\sigma_2^{-1}=(\sigma_1\sigma_2)^{6n-1},\textrm{ where } n=(t+1)/2.
\]
This is precisely the torus knot $T(3,6n-1)$ whose double branched cover is the Brieskorn sphere $\Sigma(2,3,6n-1)$ (\cite{milnor1975}), hence $S^3_{(-1)}(L_1^t)=\Sigma(2,3,6n-1)$. It follows from Heegaard Floer homology $d$-invariant calculations in~\cite{tweedy} that $S^3_{(-1)}(L_1^t)$ does not bound a $\Q B^4$.\footnote{It was originally shown in \cite{furuta} that $S^3_{(-1)}(L_1^t)$ does not bound a $\Z B^4$.}

If $t$ is even, then $S^3_{(1)}(L_1^t)$ is diffeomorphic to $\Sigma(2,3,6n+1)$, where $n=t/2$ (see, e.g., Example~1.4 in~\cite{saveliev}). 
By~\cite{akbulutkirby2}, \cite{akbulutlarson}, \cite{fickle} and~\cite{fintushelstern}, it follows that $S^3_{(1)}(L_1^t)$ bounds a $\Q B^4$ for $n \in \{1,2,3,4\}$.
\end{proof}

\begin{proof}[Proof of Proposition \ref{prop:chi2,3}]
Figure~\ref{fig:fam2} in Section~\ref{sec:ribbons} shows that all closures of braids of the form (ii) are $\chi$-ribbon. By Proposition~\ref{prop:parabolic} and Lemma~\ref{lem:dbc}, any 3-braid closure of the form~(i) with $t$ even (resp., $t$ odd) and $m\in\{-2,-3\}$ (resp., $m\in\{-1,-2\}$) is not $\chi$-slice. Suppose either $m=-1$ and $t$ is even, or $m=-3$ and $t$ is odd. In the first case, the closure is the torus knot $T(3,3t-1)$, which is known to not be slice for all $t\neq 0$; if $t=0$, we have the unknot, which is slice. In the second case, the closure of $(\sigma_1 \sigma_2)^{3t} \sigma_1^m \sigma_2^{-1}$ is a knot whose signature equals $4 - 4t$ by~\cite{erle}, hence it is not slice for $t \neq 1$; if $t=1$, then it is the unknot, which is slice.\qedhere
\end{proof}

\section{Dual strings and the sets $\mathcal{S}_i$}
\label{sec:notation}

Let $\textbf{a}=(a_1,\ldots,a_n)$, where $a_i\ge 2$ for all $1\le i\le n$ and let $a_j\ge3$ for some $j$. Recall that $Y^t_{\textbf{a}}=\Sigma_2(B^t_{\textbf{a}})$, where $B^t_{\textbf{a}}$ is the closure of the 3-braid given by

\[
    \tag{$\star$}
    \label{eq:3-braid}
    (\sigma_1 \sigma_2)^{3t} \sigma_1 \sigma_2^{-(a_1-2)} \dots \sigma_1 \sigma_2^{-(a_n-2)}.
\]

By Theorem 4.2 in \cite{baldwinquasi}, $B^t_{\textbf{a}}$ is QA if and only if $t\in\{-1,0,1\}$. 
We call $\textbf{a}$ the \textit{associated string} of $B^t_{\textbf{a}}$ and $Y^t_{\textbf{a}}$.
Since closures of such 3-braids with fixed $t$ whose associated strings are related by cyclic reorderings and reversals 
are isotopic, we only need to consider associated strings up to those two operations.

Any string of integers $(b_1, \dots, b_k)$ with $b_i \ge 2$ for all $i$ and some $b_j \ge 3$ can be written in the form $$(2^{[x_1]}, 3 + y_1, 2^{[x_2]}, 3+y_2, \dots, 2^{[x_m]}, 2 + y_m),$$ where $m\ge 1$, $x_i, y_i \ge0$ for all $i$, and $2^{[x_i]}$ denotes a substring consisting of the integer 2 repeated $x_i$ times. Given such string, we define its \emph{linear dual} to be the string $$(2+x_1, 2^{[y_1]}, 3+x_2, 2^{[y_2]}, 3+x_3, \dots, 3+x_m, 2^{[y_m]}).$$ The linear duals of the strings $(2^{[k]})$ for $k \ge 1$ and $(1)$ are defined to be $(k+1)$ and the empty string, respectively. The \emph{cyclic dual} of a string $$(2^{[x_1]}, 3+y_1, 2^{[x_2]}, 3+y_2, \dots, 2^{[x_m]}, 3+y_m)$$ with $m \ge 1$ and $x_i, y_i \ge 0$ for all $i$ is given by $$(3+x_1, 2^{[y_1]}, 3+x_2, 2^{[y_2]}, \dots, 3+x_m, 2^{[y_m]}).$$

\noindent The next two results are important in future sections.
\begin{lem}[Lemma 2.3 in \cite{simone2020classification}]
    Let $\textbf{a}$ and $\textbf{d}$ be cyclic dual strings. Then reversing the orientation of $Y_\textbf{a}^t$ yields $Y_\textbf{d}^{-t}.$
    \label{lem:revorient}
\end{lem}

\noindent On the level of the 3-braid, we
have a stronger statement.

\begin{lem} The mirror of $B^t_{\textbf{a}}$ is isotopic to $B^{-t}_{\textbf{d}}$.
\label{lem:mirror}
\end{lem}

\begin{proof}
Let $$\textbf{a}=(2^{[x_1]}, 3 + y_1, 2^{[x_2]}, 3+y_2, \dots, 2^{[x_m]}, 2 + y_m)$$
and let 
$$\textbf{d}=(3+x_1, 2^{[y_1]}, 3+x_2, 2^{[y_2]}, \dots, 3+x_m, 2^{[y_m]})
$$ be its cyclic dual.
Then $B_\textbf{a}^t$ is the closure of the 3-braid $$\beta=(\sigma_1\sigma_2)^{3t}\sigma_1^{x_1+1}\sigma_2^{-(y_1+1)}\cdots\sigma_1^{x_m+1}\sigma_2^{-(y_m+1)}.$$ The mirror $mB_\textbf{a}^t$ is the closure of the 3-braid 
\begin{align*}
    m\beta&=(\sigma_1\sigma_2)^{-3t}\sigma_1^{-(x_1+1)}\sigma_2^{y_1+1}\cdots\sigma_1^{-(x_m+1)}\sigma_2^{y_m+1}.
\end{align*}
View $mB_\textbf{a}^t$ as sitting in the $xy$-plane of $\mathbb{R}^3\subset S^3$ and wrapping around the $z$-axis such that the braided portion of the link lies in the region $\{(x,y)\,|\,x> 0\}$. Then performing a $180^\circ$ rotation about the $y$-axis and $z-$axis provides an isotopy between  $mB_\textbf{a}^t$ and the closure of the 3-braid
\begin{align*}
\beta'&=(\sigma_1\sigma_2)^{-3t}\sigma_2^{-(x_1+1)}\sigma_1^{y_1+1}\cdots\sigma_2^{-(x_m+1)}\sigma_1^{y_m+1}.
\end{align*}
Now by conjugating with $\sigma_1$, we have that this link is isotopic to the closure of the 3-braid 
\begin{align*}
\beta'&=(\sigma_1\sigma_2)^{-3t}\sigma_1\sigma_2^{-(x_1+1)}\sigma_1^{y_1+1}\cdots\sigma_2^{-(x_m+1)}\sigma_1^{y_m},
\end{align*}
which is $B_{\textbf{d}}^{-t}$.
\end{proof}

Let us now define the following sets of strings, where in each case $(b_1, \dots, b_k)$ and $(c_1, \dots, c_l)$ are linear duals of each other:
\begin{itemize}
    \item $\mathcal{S}_{1a} = \{ (b_1, \dots, b_k, 2, c_l, \dots, c_1, 2) \mid k + l \ge 3 \}$;
    \item $\mathcal{S}_{1b} = \{ (b_1, \dots, b_k, 2, c_l, \dots, c_1, 5) \mid k + l \ge 2 \}$;
    \item $\mathcal{S}_{1c} = \{ (b_1, \dots, b_k, 3, c_l, \dots, c_1, 3) \mid k + l \ge 2 \}$;
    \item $\mathcal{S}_{1d} = \{ (2, b_1 + 1, b_2, \dots, b_{k-1}, b_k + 1, 2, 2, c_l + 1, c_{l-1}, \dots, c_2, c_1 + 1, 2) \mid k + l \ge 3 \}$;
    \item $\mathcal{S}_{1e} = \{ (2, 3 + x, 2, 3, 3, 2^{[x-1]}, 3, 3) \mid x \ge 1 \}\cup\{(2,3,2,3,4,3)\}$;
    \item $\mathcal{S}_{2a} = \{ (b_1 + 3, b_2, \dots, b_k, 2, c_l, \dots, c_1) \}$;
    \item $\mathcal{S}_{2b} = \{ (3 + x, b_1, \dots, b_{k-1}, b_k + 1, 2^{[x]}, c_l + 1, c_{l-1}, \dots, c_1) \mid x \ge 0 \textrm{ and } k + l \ge 2 \}$;
    \item $\mathcal{S}_{2c} = \{ (b_1 + 1, b_2, \dots, b_{k-1}, b_k + 1, c_1, \dots, c_l) \mid k + l \geq 2 \}$;
    \item $\mathcal{S}_{2d} = \{ (2, 2+x, 2, 3, 2^{[x-1]}, 3, 4) \mid x \ge 1 \} \cup \{ (2, 2, 2, 4, 4) \}$;
    \item $\mathcal{S}_{2e} = \{ (2, b_1 + 1, b_2, \dots, b_k, 2, c_l, \dots, c_2, c_1 + 1, 2) \mid k + l \ge 3 \} \cup \{ (2,2,2,3) \} $;
    \item $\mathcal{O} = \{(6,2,2,2,6,2,2,2), (4,2,4,2,4,2,4,2), (3,3,3,3,3,3)\}$.
    \item $\mathcal{S}_1=\mathcal{S}_{1a}\cup\mathcal{S}_{1b}\cup\mathcal{S}_{1c}\cup\mathcal{S}_{1d}\cup\mathcal{S}_{1e}$
    \item $\mathcal{S}_2=\mathcal{S}_{2a}\cup\mathcal{S}_{2b}\cup\mathcal{S}_{2c}\cup\mathcal{S}_{2d}\cup\mathcal{S}_{2e}$
\end{itemize}

We further define $\mathcal{S}_{i}^*$ to be the set of cyclic dual of the strings belonging to $\mathcal{S}_i$. It is worth noting that $\mathcal{S}_{2c}^*=\mathcal{S}_{2c}$ as every string in $\mathcal{S}_{2c}$ is cyclic dual to itself; the same is true of strings in $\mathcal{O}$. Finally, recall that we denote $(3,3,3,3,3,3) \in \mathcal{O}$ by $\mathbf{3}_6$.

\section{The Cubiquity Obstruction}
\label{sec:cubiquitydef}

In this section we recall a refinement of the Donaldson's theorem obstruction to the existence of a $\QQ B^4$ bounded by a given $\QQ S^3$. The classical form of the obstruction states that if a $\QQ S^3$ bounds both a $\QQ B^4$ and a 4-manifold $X$ with negative-definite intersection form $Q_X$, then the lattice $\Lambda_X = (H_2(X;\ZZ) / \mathrm{Tors}, Q_X)$ admits an embedding $\varphi_X: \Lambda_X \hookrightarrow (\ZZ^{\rk\,\Lambda_X}, -I)$ into the negative-definite integral lattice of equal rank. In~\cite{greene-pretzel}, Greene and Jabuka have derived a more restrictive condition on such embeddings, dubbed \emph{cubiquity} in~\cite{greene-owens} and applicable when $X$ is \emph{sharp}, which is a property related to the Heegaard Floer homology of its boundary. This condition will prove fruitful in the following to obstruct the existence of $\QQ B^4$s bounded by $Y^t_{\textbf{a}}$, where either: $t=-1$ and $\textbf{a}\in\mathcal{S}_{1a}^* \cup (\mathcal{O}\setminus\{\mathbf{3}_6\})$; or $t=1$ and $\textbf{a}\in\mathcal{S}_{1a} \cup (\mathcal{O}\setminus\{\mathbf{3}_6\})$.

\subsection{Cubiquitous Subsets}

We begin with the lattice aspect of the refined obstruction. Hereafter, we denote the negative-definite integral lattice $(\ZZ^n, -I)$ simply by $\ZZ^n$. The next definition and the proposition following are due to Greene and Owens~\cite{greene-owens}.

\begin{defn}
    A subset $S \subset \ZZ^n$ is \emph{cubiquitous} if it has non-zero intersection with every unit cube in $\ZZ^n$, i.e., 
        \[
            S \cap (x + \{ 0, 1\}^n) \neq \varnothing~\textrm{for all}~x\in\ZZ^n.
        \]
    A lattice $\Lambda$ is cubiquitous if it admits an embedding into $\ZZ^{\rk\,\Lambda}$ whose image is cubiquitous; such embeddings are also called cubiquitous.
    \label{def:cubiquitous}
\end{defn}

\begin{prop}[Proposition~2.1 in~\cite{greene-owens}]
    Let $ \Lambda $ be a sublattice of $\ZZ^n$. The following conditions are equivalent:
    \begin{enumerate}
        \item $\Lambda$ is cubiquitous;
        \item every coset of $\Lambda$ is cubiquitous;
        \item every coset of $\Lambda$ contains a point of the unit cube $\{0,1\}^n$.
    \end{enumerate}
\end{prop}

Condition (3) is particularly useful as it enables us to check whether a lattice embedding is cubiquitous in the following way. Let $\Lambda$ be a lattice endowed with a fixed basis and suppose $\rk\,\Lambda = n$. Choose an orthonormal basis $\{e_1, \dots, e_n\}$ of $\ZZ^n$ and let $\varphi: \Lambda\hookrightarrow\ZZ^n$ be a lattice embedding represented with respect to the chosen bases by an integral matrix $B$. Let $D$ be the Smith normal form of $B$, i.e., the diagonal matrix $D=\mathrm{diag}(a_1,\dots,a_n) \in \mathrm{Mat}_n(\ZZ)$ such that $a_1 \geqslant 1$ and $a_i\,|\,a_{i+1}$ for $i=1,\dots, n-1$, satisfying the condition that $D=UBV$ for two matrices $U, V \in \mathrm{Mat}_n(\ZZ)$ which are invertible over $\ZZ$. Consider the commutative diagram
\begin{equation*}
    \xymatrix{
        \ZZ^n \ar[r]^{B} & \ZZ^n \ar[r] \ar[d]^{U} &\ZZ^n/B\ZZ^n \ar[d]^{\psi} \\
        \ZZ^n \ar[u]^{V} \ar[r]^{D} & \ZZ^n  \ar[r] & \ZZ^n/D\ZZ^n \ar[r]^<<<<<{\sim} & \displaystyle{\bigoplus_{i=1}^n \ZZ/a_i\ZZ},
    }
\end{equation*}
where the unlabelled arrows are canonical quotient maps and $\psi([x]) = [Ux]$ for all $[x]\in \ZZ^n/B\ZZ^n$. Every class in $\ZZ^n/D\ZZ^n$ is represented by a vector $y=(y_1,\dots,y_n)$ with $0\le y_i \le a_i$ for $i=1,\dots, n$, hence every class in $\ZZ^n/B\ZZ^n$ is represented by $ U^{-1}y $ for some such $y$. Clearly, for $z\in\{0,1\}^n$ we have that $ [U^{-1}x] = [z] $ if and only if $U^{-1}x - z \in \im\,B$. To verify that $\varphi$ is cubiquitous, it suffices to check that for every $y$ as above, there exists $z\in \{0,1\}^n$ such that $B^{-1}(U^{-1}y - z) \in \ZZ^n $, where $B^{-1}$ is the inverse of $B$ over $\QQ$. This procedure is implemented in the accompanying \textsc{SageMath} notebook, which will be used in Section~\ref{sec:proof1} to prove Theorem~\ref{thm:1}.

\subsection{Sharp Manifolds}

Suppose $Y$ is a $\QQ S^3$ equipped with a $\spinc$ structure $ \mathfrak{t} $. In~\cite{ozsz2003}, Ozsv\'ath and Szab\'o employ Heegaard Floer homology to associate to every such pair a rational number $ d(Y,\mathfrak{t})$, called the \emph{correction term}, or the \emph{$d$-invariant}. If $X$ is a negative-definite 4-manifold bounded by $Y$ and equipped with a $\spinc$ structure $\mathfrak{s}$, we have that 
    \begin{equation*}
        \tag{$\dagger$}
        \label{eq:ineq-sharp}
        c_1(\mathfrak{s})^2 + b_2(X) \le 4d(Y, \mathfrak{s}|_Y),
    \end{equation*}
where $c_1(\mathfrak{s})$ is the first Chern class of $\mathfrak{s}$, $b_2(X)$ is the second Betti number of $X$, and $\mathfrak{s}|_Y$ is the restriction of $\mathfrak{s}$ to $Y$~\cite{os-absolutelygraded}.

\begin{defn}
    A negative-definite 4-manifold $X$ with $\QQ S^3$ boundary $ Y $ is \emph{sharp} if for every $\mathfrak{t} \in \Spinc(Y)$ there exists $\mathfrak{s} \in \Spinc(X)$ with $\mathfrak{t} = \mathfrak{s}|_Y$ such that equality is attained in~\eqref{eq:ineq-sharp}.
\end{defn}

We can now state the cubiquity obstruction precisely.

\begin{thm}[Theorem~6.1 in~\cite{greene-owens}]
    \label{thm:cbobstruction}
    Let $X$ be a sharp 4-manifold with the intersection lattice $\Lambda_X$. If $\pp X$ is a $\QQ S^3$ that also bounds a $\QQ B^4$, then $\Lambda_X$ admits a cubiquitous embedding into $\ZZ^{\rk\,\Lambda_X}$.
\end{thm}

In view of the above discussion, one can show that a $\QQ S^3$ does not bound a $\QQ B^4$ by constructing a sharp 4-manifold $X$ bounded by the $\QQ S^3$, finding all embeddings of $\Lambda_X$ into the standard integral lattice of the same rank, and verifying that none of them are cubiquitous.

\section{Sharp Manifolds and QA 3-Braid Closures}
\label{sec:sharpdbc}

Let $\textbf{a}=(a_1,\ldots,a_n)$, where $a_i\ge 2$ for all $i$. Let $X_{\textbf{a}}^t$ denote the 4-manifold with handlebody diagram given in Figure \ref{fig:P}; recall that $t$ indicates the number of half-twists.
Note that if $a_j\ge 3$ for some $j$, then $\partial X_{\textbf{a}}^t=Y^t_{\textbf{a}}$. 
Note that if $\textbf{a}'$ is any cyclic reordering of $\textbf{a}$, then $X_{\textbf{a}}^t$ and $X_{\textbf{a}'}^t$ are diffeomorphic.
As discussed in Section \ref{sec:dbcs}, $Y_{\textbf{a}}^t$ is the double cover of $S^3$ branched over the closure of the 3-braid $(\sigma_1\sigma_2)^{3t}\sigma_1\sigma_2^{-(a_1-2)}\cdots\sigma_1\sigma_2^{-(a_n-2)}$. Note that it follows from Theorem~4.1 in~\cite{baldwinquasi} that when $t\in\{-1,0,1\}$, $Y^t_{\textbf{a}}$ is an $L$-space. The goal of this section is to prove the following:

\begin{thm} Let $\textbf{a}=(a_1,\ldots,a_n)$ such that $a_i\ge 2$ for all $i$ and $a_j\ge 3$ for some $j$. Then $X_{\textbf{a}}^t$ is sharp if and only if $t$ is even or $t\le0$ is odd.
\label{thm:sharp}
\end{thm}

\begin{figure}[h]
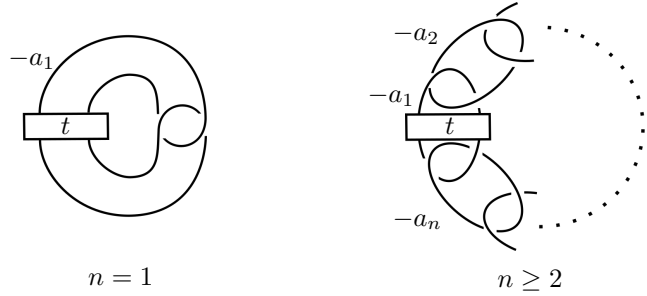

\centering
\begin{overpic}
[scale=.6]{figures/n1.png}
\put(-5,75){$-a_1$}
\put(16,47){$t$}
\put(27,-14){$n=1$}
\end{overpic}\hspace{1in}
\begin{overpic}
[scale=.6]{figures/n_1.png}
\put(-15,60){$-a_1$}
\put(-5,85){$-a_2$}
\put(-5,10){$-a_n$}
\put(16,47){$t$}
\put(37,-15){$n\ge2$}
\end{overpic}\vspace{.4cm}
\caption{The 4-manifold $X^t_{(a_1, \dots, a_n)}$ whose boundary is $Y^t_{(a_1, \dots, a_n)}$, the double cover of $S^3$ branched over the closure of the 3-braid $(\sigma_1\sigma_2)^{3t}\sigma_1\sigma_2^{-(a_1-2)}\dots\sigma_1\sigma_2^{-(a_n-2)}$.}\label{fig:P}
\end{figure}

Although we will only need the sharpness of $X^{-1}_{\textbf{a}}$, we will prove the much more general result as it might be of independent interest. To prove Theorem~\ref{thm:sharp}, we will use induction. To this end, we start with the base cases.

\begin{lem} Let $n\ge2$ and $a_i=2$ for all $i$. Then $X^{-1}_{\textbf{a}}$ is sharp.
\label{lem:basesharp}
\end{lem}

\begin{proof}
Set $X=X_{\textbf{a}}^{-1}$ and $Y=\partial X$. 
Let $Q$ denote the intersection form of $X$.
It is easy to see that if $n\ge2$, then $|\det Q|=4$; hence $|H_1(Y)|=|\Spinc(Y)|=4$.
Moreover, note that $Y$ is the double cover of $S^3$ branched over the closure of the 3-braid $(\sigma_1\sigma_2)^{-3}\sigma_1^n$.

We claim that the $d$-invariants of $Y$ are $\{\frac{n}{4}-1,\frac{n}{4},0,0\}$. 
If $n=2$, then $Y=L(2,1)\#\allowbreak L(2,1)$; by Theorem~4.3 and Proposition~4.8 in~\cite{os-absolutelygraded}, the $d$-invariants are indeed $\{-\frac{1}{2},\frac{1}{2},\allowbreak 0,0\}$. Now assume $n\ge3$.
By Theorem~6.2(2) in~\cite{baldwinquasi}, there is a $\spinc$ structure $\mathfrak{s}_0$ satisfying $d(Y,\mathfrak{s}_0)=\frac{n}{4}-1$.

To show that there is a $\spinc$ structure $\mathfrak{t}$ such that $d(Y,\mathfrak{t})=\frac{n}{4}$, we will construct a negative-definite plumbing $Z$ with $\partial Z=Y$ and $b_2(Z)=n$, and use the method of \cite{ozsvathszaboplumbed}. Namely, we will find a characteristic element $K\in H^2(Z)$ such that $\frac{K^2+n}{4}=\frac{n}{4}$, or $K^2=0$, and that satisfies the following: 
if $K=c_1(\mathfrak{s})$ and $\mathfrak{s}|_{Y}=\mathfrak{s'}|_{Y}$ for some $\spinc$ structure $\mathfrak{s}'$ on $Z$, then $K^2\ge c_1^2(\mathfrak{s'})$.
Consider Figure \ref{fig:X-1sharpbase}. The first handlebody diagram is that of $X$. Blow up the diagram with a $+1$-framed unknot as in the second diagram. We can then blow down $n-3$ successive $-1$-framed unknots to obtain the third diagram. After handle sliding as indicated by the green arrow, we obtain the fourth diagram. Finally, blow up the linking between the $(n-2)$-framed and $-1$-framed 2-handles with a $+1$-framed unknot and perform successive blowdowns until we obtain the last diagram; call the resulting 4-manifold $Z$. Note that $\partial Z=Y$. By \cite{ozsvathszaboplumbed}, $Z$ is sharp. 
Since the framing of each 2-handle of $Z$ is even, the class $K=0$ is characteristic in $H^2(Z)$. Hence $K^2=0$. Since $Z$ is negative-definite, if $K=c_1(\mathfrak{s})$ and $\mathfrak{s}|_{Y}=\mathfrak{s'}|_{Y}$ for some $\spinc$ structure $\mathfrak{s}'$ on $Z$, then $K^2\ge c_1(\mathfrak{s}')^2$. Hence $d(Y,\mathfrak{s}|_Y)=\frac{n}{4}$. 

It is easy to see that $Y$ bounds a $\QQ B^4$: blow down the third diagram in Figure \ref{fig:X-1sharpbase} two times and then change the resulting $0$-framed unknot into a dotted circle, as shown in Figure \ref{fig:ball}, to see a $\QQ B^4$ bounded by $Y$. 
Hence there is a metaboliser of $\spinc$ structures for which the $d$-invariant vanishes (cf.~Section~2.3 in~\cite{greene-pretzel}). 
Thus the remaining two $\spinc$ structures must have vanishing $d$-invariants.

\begin{figure}
    \centering
    \hspace{.5cm}
    \begin{overpic}
    [scale=.5]{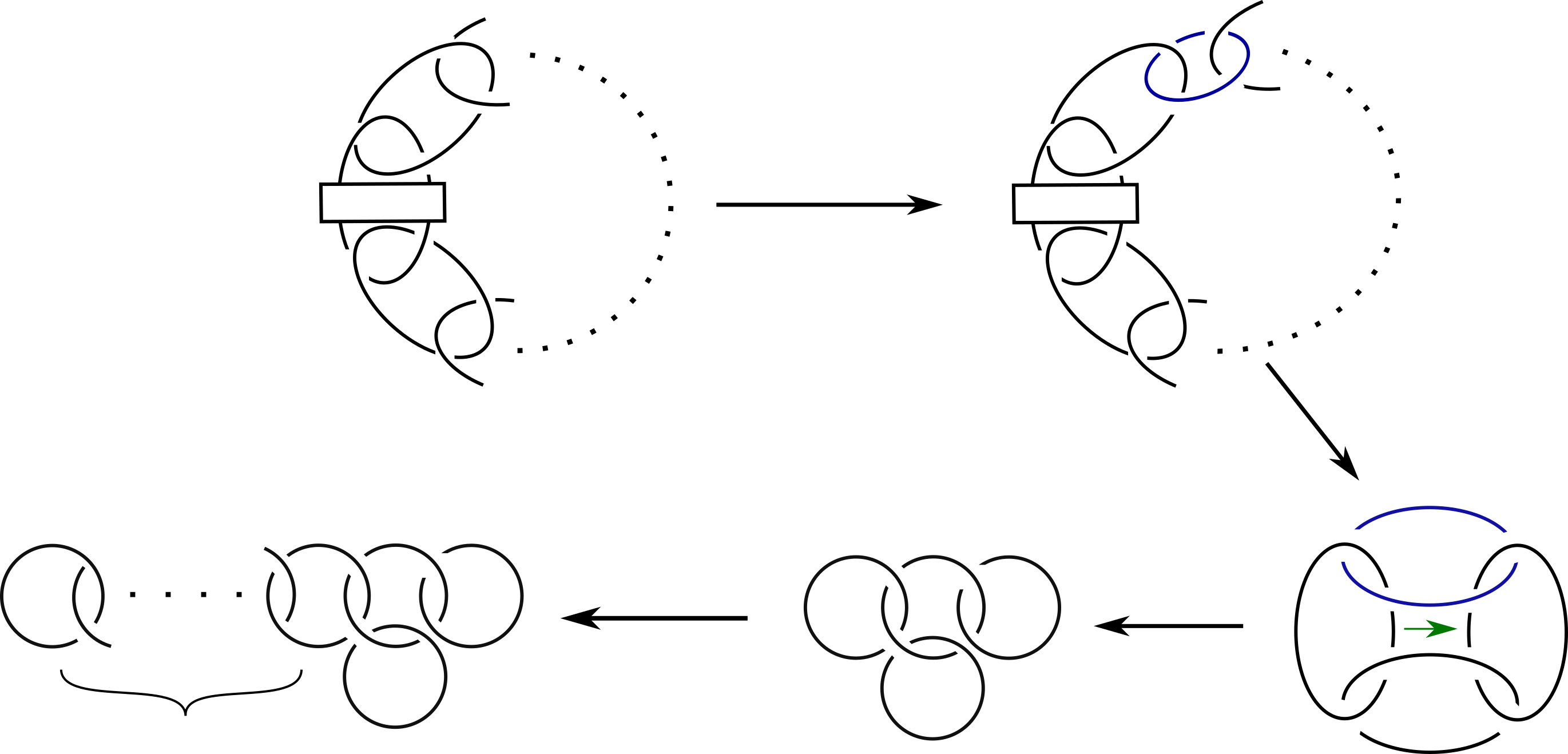}
    \put(2,35){\huge$X=$}
    \put(21,25){$-2$}
    \put(17,38){$-2$}
    \put(21,44){$-2$}
    \put(22,34.25){$-1$}
    \put(45,37){blow down}
    \put(65,25){$-2$}
    \put(61.5,38){$-2$}
    \put(65,44){$-1$}
    \put(75,47){\color{darkblue}$1$}
    \put(78,48.5){$-1$}
    \put(66,34.2){$-1$}
    \put(86,20){\shortstack[c]{blow down\\$n-3$ times}}
    \put(88.5,16.5){\color{darkblue}$n-2$}
    \put(79,10){$-1$}
    \put(100,10){$-1$}
    \put(89,-3){$-2$}
    \put(71,10){slide}
    \put(49,13.5){\color{darkblue}$n-2$}
    \put(58,13.5){$-1$}
    \put(63,13.5){$-2$}
    \put(61,0){$-2$}
    \put(37,2.5){\shortstack[c]{blow up\\ and\\\\\\blow\\down}}
    \put(28,14){$-2$}
    \put(22,14){$-2$}
    \put(17,14){$-2$}
    \put(1,14){$-2$}
    \put(26,0){$-2$}
    \put(8,0){$n-3$}
    \put(-10,8.2){\huge$Z=$}
    \end{overpic}
    \vspace{.2cm}
    \caption{$\partial X=\partial Z$.}
    \label{fig:X-1sharpbase}
\end{figure}

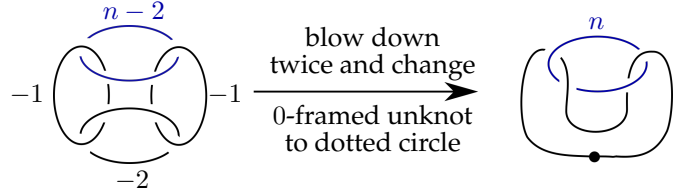
\begin{figure}
    \centering
    \begin{overpic}
    [scale=.5]{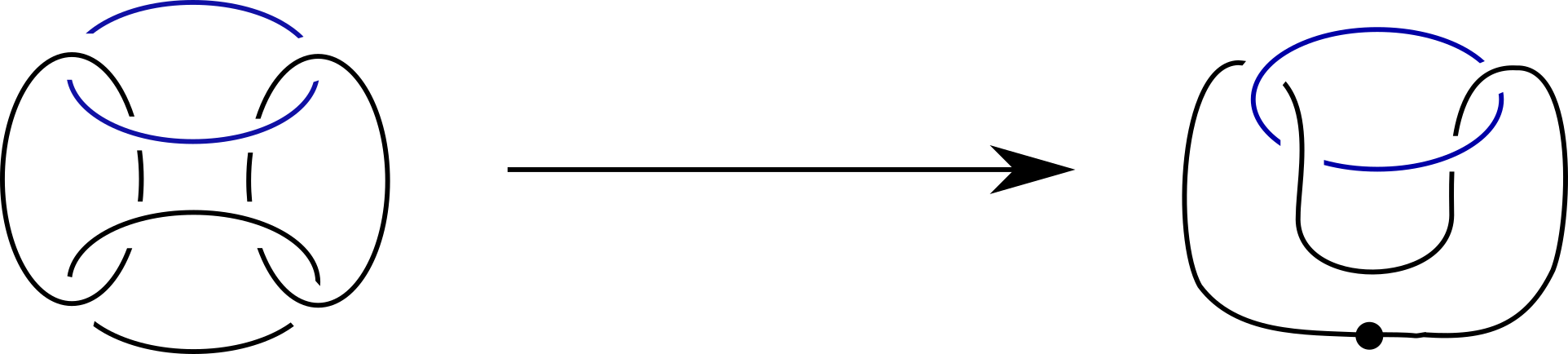}
    \put(8,23.5){\color{darkblue}$n-2$}
    \put(-7,10){$-1$}
    \put(25,10){$-1$}
    \put(10,-4){$-2$}
    \put(35,2){\shortstack[c]{blow down\\ twice and change\\\\\\$0$-framed unknot\\to dotted circle}}
    \put(86.5,22){\color{darkblue}$n$}
    \end{overpic}
    \vspace{.2cm}
    \caption{$Y$ bounds a rational homology ball.}
    \label{fig:ball}
\end{figure}

It remains to show that for each $\spinc$ structure $\mathfrak{t}$ on $Y$, there exists a $\spinc$ structure $\mathfrak{s}$ on $X$ such that $\mathfrak{s}|_Y=\mathfrak{t}$ and $c_1(\mathfrak{s})^2 + b_2(X)= 4d(Y, \mathfrak{t})$, or $c_1(\mathfrak{s})^2= 4d(Y, \mathfrak{t})-n\in\{-n,-n,-4,0\}$. Thus we need to find characteristic elements $K_1,K_2,K_3,K_4\in H^2(X)$ whose respective squares are $0,-4,-n,$ and $-n$, and whose corresponding $\spinc$ structures $\mathfrak{s}_i$ for $1\le i\le 4$ satisfy $\mathfrak{s}_i|_{Y}\neq\mathfrak{s}_j|_{Y}$ for $i\neq j$. Set
\begin{align*}
    K_1=(0,\ldots,0)^T, &\quad K_2=(2,0,\ldots,0,2)^T,\\ \quad K_3=(2,0,\ldots,0)^T, &\quad K_4=(0,2,0,\ldots,0)^T.
\end{align*}
Computing $K_i^2=K_i^TQ_X^{-1}K_i$ yields $K_1^2=0$, $K_2^2=-4$, and $K_3^2=K_4^2=-n$. 
Let $\mathfrak{s}_1,\mathfrak{s}_2,\mathfrak{s}_3,$and $\mathfrak{s}_4$ be the unique $\spinc$ structures on $X$ satisfying $c_1(\mathfrak{s}_i)=K_i$ for $1\le i\le 4$. Recall that $\spinc$ structures on $Y$ are in a one-to-one correspondence with $2H^2(X,Y)$-orbits in the set of characteristic elements in $H^2(X)$; hence if $\mathfrak{s}_i=\mathfrak{s}_j$, then $K_i-K_j\in 2\,\im(Q)$, where $\im(Q)$ is the image of $Q$, viewed as a map $H^2(X,Y)\to H^2(X)$. It is easy to check that $\frac{1}{2}Q^{-1}(K_i-K_j)\not\in \ZZ^n$ for all $i\neq j$; consequently, $\mathfrak{s}_1|_Y,\mathfrak{s}_2|_Y,\mathfrak{s}_3|_Y,$ and $\mathfrak{s}_4|_Y$ are the four distinct $\spinc$ structures on $Y$. Hence $X$ is sharp.
\end{proof}

\begin{lem} Let $n\ge2$, $a_1=3$, and $a_i=2$ for all $i\neq1$. Then $X^{0}_{\textbf{a}}$ is sharp.
\label{lem:basesharp2}
\end{lem}

\begin{proof}
This follows in the same way as the proof of Lemma~\ref{lem:basesharp}. First notice that $Y^{0}_{\textbf{a}}$ is a lens space; indeed, by blowing up the obvious surgery diagram of $Y^{0}_{\textbf{a}}$ between the $-3$-framed unknot and an adjacent $-2$-framed unknot and then performing $n+1$ successive blowdowns, we obtain a surgery diagram consisting of a single unknot with framing $n$. Thus by using Proposition 4.8 in~\cite{os-absolutelygraded}, the $d$-invariants of $Y^{0}_{\textbf{a}}$ are $$\left\{\frac{-n+(2i-n)^2}{4n}\,\Big|\,0\le i< n\right\}.$$ 
As in the proof of Lemma~\ref{lem:basesharp}, we must find characteristic elements in $H^2(X_{\textbf{a}}^0)$ that square to the values in the set
$$D=\left\{\frac{4i^2}{n}-4i-1\,\Big|\,0\le i< n\right\}.$$
Consider the vectors $K_j=e_1+\sum_{i=2}^j(-1)^{i-1}2e_i$, where $1\le j\le n$ and $\{e_1,\ldots,e_n\}$ is the standard basis for $\ZZ^n$. Following as in the proof of Lemma~\ref{lem:basesharp}, it can be shown that $K_j^2\in D$ for all $j$ and that these vectors correspond to spin$^c$ structures that restrict to distinct spin$^c$ structures on $Y_{\textbf{a}}^0$. 
The result follows.
\end{proof}

\begin{definition}
Let $M$ be an oriented 3-manifold with torus boundary, and let $\gamma_0,\gamma_1,\gamma_2$ be simple closed curves in $\partial M$ 
such that
$$\#(\gamma_0\cap \gamma_1)=\#(\gamma_1\cap \gamma_2)=\#(\gamma_2\cap \gamma_0)=-1,$$ where $\#$ denotes algebraic intersection number and the orientation of $\pp M$ is induced by that of $M$.
Let $Y_i$ denote the 3-manifold obtained by gluing a solid torus to $M$ such that the meridian of the boundary of the solid torus is identified with $\gamma_i\subset \partial M$ for $i\in\{0,1,2\}$. Then $(Y_0,Y_1,Y_2)$ is called a \textit{surgery triad}.
\end{definition}

\begin{thm}[Theorem 2.2 in~\cite{ozsvathszabosharp}]
Let $(Y_0,Y_1,Y_2)$ be a surgery triad. Then there exists a long exact sequence 
$$\cdots\to HF^+(Y_0)\to HF^+(Y_1)\to HF^+(Y_2)\to\cdots$$
where the maps are induced from the obvious 2-handle cobordisms connecting $Y_i$ to $Y_{i+1}$, where $i\in\ZZ/3$.
\label{thm:triad}
\end{thm}

\begin{prop}[Proposition 2.6 in~\cite{ozsvathszabosharp}] 
Suppose $(Y_0,Y_1,Y_2)$ is a triple of $\QQ S^3$s that form a surgery triad such that $Y_0$ and $Y_2$ are $L$-spaces. Let $W_i:Y_i\to Y_{i+1}$ denote the 2-handle cobordism connecting $Y_i$ to $Y_{i+1}$. If $-Y_2$ bounds a sharp 4-manifold $X_2$ and $X_0=X_2 \cup (-W_1) \cup (-W_0)$ is sharp, then $X_1=X_2\cup (-W_1)$ is also sharp.
\label{prop:os}
\end{prop}

\begin{remark}
Note that our orientation conventions differ from the conventions used in~\cite{ozsvathszabosharp}. As a result, we adapted the statement of Proposition 2.6 in~\cite{ozsvathszabosharp} to our conventions.
\end{remark}

Given a sequence of non-zero integers $(a_1,\ldots,a_n)$, their \emph{(Hirzebruch--Jung) continued fraction expansion} is given by  
\[
    [a_1,\ldots,a_n] = a_1 - \cfrac{1}{a_2 - \cfrac{1}{\dots - \cfrac{1}{a_n}}}.
\]
Given coprime integers $p>q\ge1$, there is a unique continued fraction expansion $[a_1,\ldots,a_n]=\frac{p}{q}$, where $a_i\ge2$ for all $i$.

\begin{proof}[Proof of Theorem~\ref{thm:sharp}]
We first assume that $t\in\{-1,0\}$.
If $n=1$, then $X^{t}_{\textbf{a}}$ is obtained by attaching a single 2-handle to $B^4$ along an unknot with framing $a_1\ge3$ (see Figure \ref{fig:P}). Hence by \cite{ozsvathszaboplumbed}, $X^{t}_{\textbf{a}}$ is sharp.

We now assume that $n\ge2$. 
We will prove sharpness by using Theorem~\ref{thm:triad}, Proposition~\ref{prop:os}, and induction. First recall that $\partial X_{\textbf{a}}^{t}$ is an $L$-space for all $\textbf{a}$. If $a_i=2$ for all $i$, then $X^{-1}_{\textbf{a}}$ is sharp by Lemma~\ref{lem:basesharp}; if $a_j=3$ for some integer $j$ and $a_i=2$ for all $i\neq j$, then $X^{0}_{\textbf{a}}$ is sharp by Lemma~\ref{lem:basesharp2} (up to cyclic reordering).

Let $\textbf{a'}=(a_1,\ldots,a_{i-1},a_i-1,a_{i+1},\ldots,a_n)$ be arbitrary and inductively assume that $X^{t}_{\textbf{a'}}$ is sharp; up to cyclic reordering, we may assume that $i=1$. 
We will show that $X^{t}_{\textbf{a}}$ is sharp, where $\textbf{a}=(a_1,\ldots,a_n)$. Let $\frac{p}{q}=[a_2,\ldots,a_n]$. We first claim that $(Y^t_{\textbf{a'}},Y^t_{\textbf{a}},L(p,q))$ forms a surgery triad. 
Let $m$ be a meridian of the $a_1-1$ surgery curve in the obvious surgery diagram of $Y^t_{\textbf{a'}}$ and let $T=\partial\nu(m)$. Then $M=Y^t_{\textbf{a'}}\setminus\mathring{\nu}(m)$ is a 3-manifold with torus boundary. Let $\gamma_2$ be the simple closed curve on $T$ that can be identified with the blackboard framing curve of $m$; let $\gamma_0$ be the simple closed curve on $T$ that bounds a disk in $\nu(m)$, oriented so that $\#(\gamma_2,\gamma_0)=-1$; and let $\gamma_1$ be the simple closed curve on $T$ satisfying $[\gamma_1]=-[\gamma_0]-[\gamma_2]\in H_2(T)$ (see Figure \ref{fig:triad}). Then $\gamma_0,\gamma_1$, and $\gamma_2$ satisfy the conditions of Theorem~\ref{thm:triad}. Moreover, using the notation of Theorem~\ref{thm:triad}, $Y_0$ is obtained by $\infty$-surgery on $m$, $Y_1$ is obtained by $1$-surgery on $m$, and $Y_2$ is obtained by $0$-surgery on $m$; hence $Y_0=Y^t_{\textbf{a}'}$, $Y_1=Y^t_{\textbf{a}}$, and $Y_2=L(p,q)$. We have thus shown that $(Y^t_{\textbf{a'}},Y^t_{\textbf{a}},L(p,q))$ forms a surgery triad.

\begin{figure}[h]
	\centering
 \begin{overpic}
 {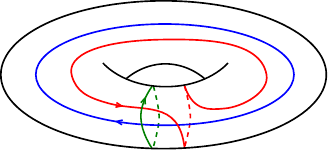}
 \put(38,2){\color{green}{$\gamma_0$}}
 \put(20,6){\color{blue}{$\gamma_2$}}
 \put(74,21){\color{red}{$\gamma_1$}}
 \end{overpic}
 \caption{Curves on $T$ defining a surgery triad.}
 \label{fig:triad}
\end{figure}

Figure \ref{fig:cobs} shows the $2$-handle cobordisms $W_i:Y_i\to Y_{i+1}$ for $i\in\{0,1\}$ inducing the long exact sequence maps in Theorem~\ref{thm:triad}. Following Section 5.5 in~\cite{gompfstipsicz}, the bottom boundary component $\partial_{-}W_i$ of $W_i$ (for $i=0,1$) has surgery diagram given by the black link and whose framings are in angle brackets. The blue framed knot denotes a 2-handle attached to $\partial_-W_i\times[0,1]$. The top boundary component $\partial_+W_i$ of $W_i$ has surgery given by the full diagram (i.e., the diagram obtained by ignoring the angle brackets). Hence it is clear, after performing blowdowns, that $\partial_-W_0=Y^t_{\textbf{a}'}$, $\partial_+W_0=Y^t_{\textbf{a}}$, $\partial_-W_1=Y^t_{\textbf{a}}$, and $\partial_+W_1=L(p,q)$, where $\frac{p}{q}=[a_2,\ldots,a_n]$. Note that $L(p,q)$ bounds a linear plumbing $X_2$ with weights $a_2,\ldots,a_n$, which is sharp by \cite{ozsvathszaboplumbed}. 
We claim that $X_{\textbf{a}}^{t}=(-W_1)\cup X_2$. If we flip the handlebody diagram of $W_1$ upside down and reverse its orientation, we obtain the first handlebody diagram in Figure \ref{fig:-cob1} (cf.~Section 5.5 in~\cite{gompfstipsicz}). Blowing down the first $\langle 1\rangle$-framed unknot yields the next diagram in Figure \ref{fig:-cob1}. Finally, after sliding the $-1$-framed blue $2$-handle over the $\langle a_1+1\rangle$-framed unknot, we obtain the last diagram in Figure \ref{fig:-cob1}. With this description, it is clear that $X_{\textbf{a}}^t=(-W_1)\cup X_2$. 

\begin{figure}
    \centering
    \begin{overpic}
    [scale=.5]{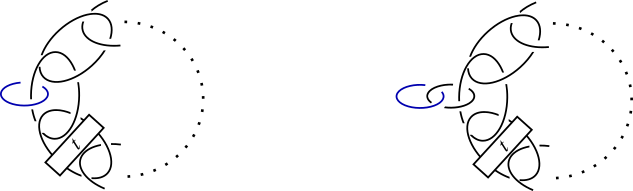}
    \put(18,12){\huge$W_0$}
    \put(-2,12){\color{darkblue}$1$}
    \put(-3,2){$\langle -a_n\rangle$}
    \put(-15,19){$\langle -(a_1-1)\rangle$}
    \put(-1,26){$\langle -a_2\rangle$}
    \put(85,12){\huge$W_1$}
    \put(63,9.5){\color{darkblue}$1$}
    \put(67.5,9.5){$\langle1\rangle$}
    \put(65,2){$\langle -a_n\rangle$}
    \put(53,19){$\langle -(a_1-1)\rangle$}
    \put(67,26){$\langle -a_2\rangle$}
    \end{overpic}
    \caption{The cobordisms $W_0$ and $W_1$.}
    \label{fig:cobs}
\end{figure}
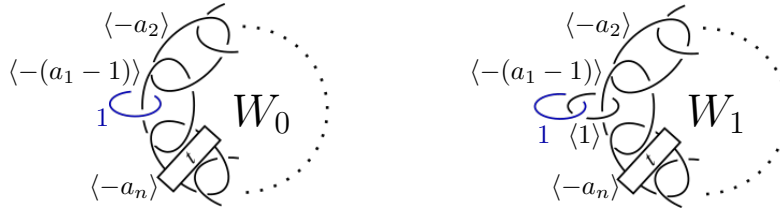

Next, consider the handlebody diagram for $-W_0$ as shown in the left side of Figure~\ref{fig:-cob0}. Blowing down the $\langle 1\rangle$-framed unknot yields the right handlebody diagram for $-W_0$ shown in Figure \ref{fig:-cob0}. Notice that the bottom boundary of $-W_0$ is $\partial X_2$; indeed if we remove the $-1$-framed $2$-handle, we are left with the surgery diagram for $\partial X_{\textbf{a}}^{t}$.  Let $X_0:=(-W_0)\cup (-W_1)\cup X_2$; note that $X_0$ has the handlebody diagram given by the right diagram in Figure \ref{fig:-cob0}, except with the brackets removed from the framings. It is thus clear that $X_0=X_{\textbf{a'}}^{t}\#\overline{\mathbb{CP}^2}$. 
By the inductive hypothesis, $X_{\textbf{a'}}^{t}$ is sharp (see, for example, \cite{ozsvathszabosharp}); hence $X_0$ is also sharp. Thus by Proposition~\ref{prop:os}, $X_{\textbf{a}}^t$ is sharp.

Now let $t$ be arbitrary. Notice that for fixed $\textbf{a}$, the 4-manifolds $X^{2k+1}_{\textbf{a}}$ (for $k\in\mathbb{Z}$) all have the same intersection form and, similarly, the 4-manifolds $X^{2k}_{\textbf{a}}$ (for $k\in\mathbb{Z}$) all have the same intersection form.
In~\cite{baldwinquasi}, Baldwin considers the $\spinc$ structure $\mathfrak{t}_0$ on $Y^t_{\textbf{a}}$ associated to a certain contact structure. In particular, he shows in Theorem 6.2 in~\cite{baldwinquasi} that 
\[
d(Y^t_{\textbf{a}},\mathfrak{t}_0)=
\begin{cases}
(3n-\sum_{i=1}^na_i)/4 \hspace{.2cm}\text{ if } \hspace{.2cm}t\text{ is even}\\
-1+(3n-\sum_{i=1}^na_i)/4 \hspace{.2cm}\text{ if } \hspace{.2cm}t<0\text{ is odd}\\
1+(3n-\sum_{i=1}^na_i)/4 \hspace{.2cm}\text{ if } \hspace{.2cm}t>0\text{ is odd}.\\
\end{cases}
\]
Moreover,  by the remarks preceding Proposition 5.1 in~\cite{baldwinquasi}, for all $i\in\mathbb{Z},$ we have the following relationship between $d-$invariants.
$$\left\{d(Y^{t}_{\textbf{a}},\mathfrak{t})\,\middle|\,\mathfrak{t}\ne\mathfrak{t}_0\right\}=\left\{d(Y^{t+2i}_{\textbf{a}},\mathfrak{t})\,\middle|\,\mathfrak{t}\ne\mathfrak{t}_0\right\}.$$

Now consider $X^{2k+1}_{\textbf{a}}$. Since $X^{-1}_{\textbf{a}}$ is sharp, it follows that $$d(Y^{-1}_{\textbf{a}},\mathfrak{t}_0)=\max_{\substack{\mathfrak{s}\in\spinc(X^{-1}_{\textbf{a}})\\\mathfrak{s}|_Y=\mathfrak{t}_0}}\frac{c_1(\mathfrak{s})^2+n}{4}.$$ 
Since the intersection form of $X^{2k+1}_{\textbf{a}}$ is the same as the intersection form of $X^{-1}_{\textbf{a}}$ for all $k$, it follows that $X^{2k+1}_{\textbf{a}}$ is sharp if and only if 
$$\left\{d(Y^{-1}_{\textbf{a}},\mathfrak{t})\,\middle|\,\mathfrak{t}\in\text{Spin}^c(Y_\textbf{a}^{-1})\right\}=\left\{d(Y^{2k+1}_{\textbf{a}},\mathfrak{t})\,\middle|\,\mathfrak{t}\ne\in\text{Spin}^c(Y_\textbf{a}^{2k+1})\right\}.$$
By the $d-$invariant calculations given above, verifying this equality reduces to  verifying 
$d(Y^{-1}_{\textbf{a}},\mathfrak{t}_0)=d(Y^{2k+1}_{\textbf{a}},\mathfrak{t}_0)$. This occurs if and only if $k<0$. Hence $X^{2k+1}_{\textbf{a}}$ is sharp if and only if $k<0$.
A similar argument shows that $X^{2k}_{\textbf{a}}$ is sharp for all $k$.
\end{proof}

\begin{figure}
    \centering
    \begin{overpic}
    [scale=.5]{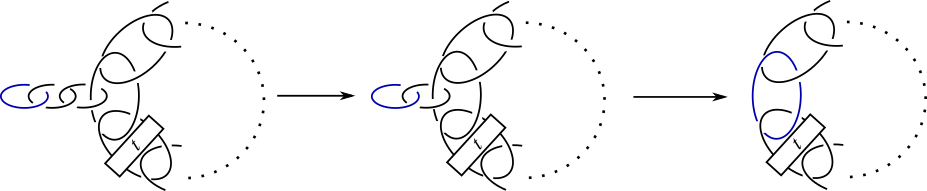}
    \put(17,8){\huge$-W_1$}
    \put(0,6){\color{darkblue}$0$}
    \put(3,6){$\langle 1\rangle$}
    \put(7,6){$\langle 1\rangle$}
    \put(4,1){$\langle -a_n\rangle$}
    \put(-3,13){$\langle -(a_1-1)\rangle$}
    \put(6,18){$\langle -a_2\rangle$}
    \put(30.5,4.5){\shortstack[c]{blow\\down}}
    
    \put(54,8){\huge$-W_1$}
    \put(39,6){\color{darkblue}$-1$}
    \put(43,6){$\langle 0\rangle$}
    \put(41,1){$\langle -a_n\rangle$}
    \put(33.5,13){$\langle -(a_1-1)\rangle$}
    \put(43,18){$\langle -a_2\rangle$}
     \put(66,4.5){\shortstack[c]{slide\\and cancel}}
     
    \put(88,8){\huge$-W_1$}
    \put(76.5,1){$\langle -a_n\rangle$}
    \put(76.5,13){\color{darkblue}{$-a_1$}}
    \put(78,18){$\langle -a_2\rangle$}
    \end{overpic}
  \caption{The cobordism $-W_1$.}
    \label{fig:-cob1}
\end{figure}

\begin{figure}
\centering
    \begin{overpic}
    [scale=.5]{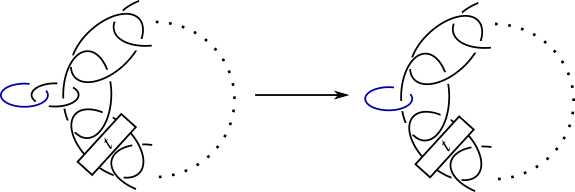}
    \put(22,14){\huge$-W_0$}
    \put(0,10.5){\color{darkblue}$0$}
    \put(4,10.5){$\langle 1\rangle$}
    \put(2,3){$\langle -a_n\rangle$}
    \put(-11,21){$\langle -(a_1-1)\rangle$}
    \put(3.5,27){$\langle -a_2\rangle$}
    \put(47,7){\shortstack[c]{blow\\down}}
    
    \put(82,14){\huge$-W_0$}
    \put(60,11){\color{darkblue}$-1$}
    \put(59.5,3){$\langle -a_n\rangle$}
    \put(59.5,21){$\langle -a_1\rangle$}
    \put(63,27){$\langle -a_2\rangle$}
    \end{overpic}
    \caption{The cobordism $-W_0$.}
    \label{fig:-cob0}
\end{figure}
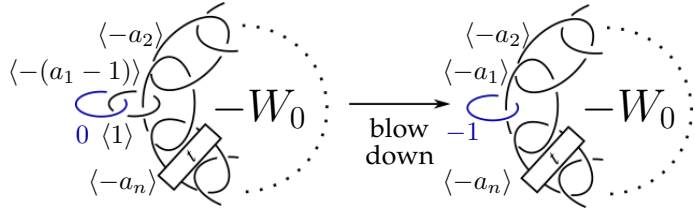

\section{Good, Standard, and Cyclic Subsets}
\label{sec:latticebasic}

In this section we establish some fundamental definitions pertaining to several classes of finite subsets of $\ZZ^n$ that shall be used in the following sections. Consider the standard negative-definite intersection lattice $(\ZZ^n,-I)$ and let $\{e_1,\ldots,e_n\}$ be an orthonormal basis of $\ZZ^n$. Then, with respect to the product $\cdot$ given by $-I$, we have $e_i\cdot e_j=-\delta_{ij}$ for all $i,j$; unless indicated otherwise, we use this product in the remainder of the paper. We begin by recalling definitions and results from~\cite{lisca-rational} and~\cite{simone2020classification}.

Given a subset $S=\{v_1,\ldots,v_n\}\subset\ZZ^n$, the \emph{intersection graph} of $S$ is the weighted graph consisting of a vertex with weight $v_i\cdot v_i$ for each vector $v_i$, and an edge labeled $v_i\cdot v_j$ between each pair of vertices $v_i$ and $v_j$ with $v_i\cdot v_j\neq0$.
We consider two subsets $S_1, S_2\subset\ZZ^n$ to be the same if $S_2$ can be obtained by applying an element of $\Aut\ZZ^n$ to $S_1$.
Let $S=\{v_1,\ldots,v_n\}\subset\ZZ^n$ be a subset. 
We call the string of integers $(a_1,\ldots,a_n)$ defined by $a_i=-v_i\cdot v_i$ the \textit{string associated} to $S$.
Two vectors $z,w\in S$ are called \textit{linked} if there exists $e\in \ZZ^n$ such that $e\cdot e=-1$ and $z\cdot e, w\cdot e\neq0$. A subset $S$ is called \textit{irreducible} if for every pair of vectors $v,w\in S$, there exists a finite sequence of vectors $v_1=v,v_2,\ldots, v_k=w\in S$ such that $v_i$ and $v_{i+1}$ are linked for all $1\le i\le k-1$; otherwise $S$ is called \emph{reducible}.

\begin{definition} A subset $S=\{v_1,\ldots,v_n\}\subset\mathbb{Z}^n$ is:
	\begin{itemize}
		\item \textit{good} if it is irreducible and $v_i\cdot v_j=
		\begin{cases}
		-a_i\le-2 & \text{if } i=j \\
		0 \text{ or } 1 & \text{if } |i-j|=1 \\
		0 & \text{otherwise;}
		\end{cases}
		$
		\vspace{.2cm}
		\item \textit{standard} if $v_i\cdot v_j= 
		\begin{cases}
		-a_i\le-2 & \text{if } i=j \\
		1 & \text{if } |i-j|=1 \\
		0 & \text{otherwise.}
		\end{cases}
		$
\end{itemize}
\end{definition}

\begin{definition} A subset $S=\{v_1,\ldots,v_n\}\subset\mathbb{Z}^n$ is:
\begin{itemize}
\item  \textit{negative cyclic} if either 
\begin{itemize}
      \item[(1)] $n=2$ and $v_i\cdot v_j=\begin{cases}
      -a_i\le-2 & \text{if } i=j \\
      0 & \text{if } i\neq j \\
      \end{cases}$ \hspace{.5cm} or \vspace{.2cm}
      \item[(2)] $n\ge3$ and there is a cyclic reordering of $S$ such that\\ $v_i\cdot v_j=\begin{cases}
      -a_i\le-2 & \text{if } i=j \\
      1 & \text{if } |i-j|=1 \\
      -1 & \text{if } i\neq j\in\{1,n\}\\
      0 & \text{otherwise;}
      \end{cases}$  \vspace{.2cm}
\end{itemize}
\item  \textit{positive cyclic} if $-a_k\le -3$ for some $k$ and either 
\begin{itemize}
      \item[(1)] $n=2$ and $v_i\cdot v_j=\begin{cases}
      -a_i\le-2 & \text{if } i=j \\
      2 & \text{if } i\neq j \\
      \end{cases}$ \hspace{.5cm} or \vspace{.2cm}
      \item[(2)] $n\ge3$ and there is a cyclic reordering of $S$ such that\\ $v_i\cdot v_j=\begin{cases}
      -a_i\le-2 & \text{if } i=j \\
      1 & \text{if } |i-j|=1 \\
      1 & \text{if } i\neq j\in\{1,n\}\\
      0 & \text{otherwise;}
      \end{cases}$   \vspace{.2cm}
\end{itemize}
\item \textit{cyclic} if $S$ is negative or positive cyclic. 
\end{itemize}
Finally, the indices of vertices are understood to be defined modulo $n$ (e.g., $v_{n+1}=v_1$).
\end{definition}

In the following, we will say that a subset is cubiquitous to mean that it generates a cubiquitous sublattice of $\ZZ^n$. Recall that a unit cube $C$ in $\ZZ^n$ is of the form $C=x+\{0,1\}^n$, where $x\in\ZZ^n$. 
Given two vectors $x,y\in\RR^n$, we define $d(x,y)$ to be the Euclidean distance between $x$ and $y$. Moreover, $||x||$ denotes the length of $x$ and $\langle x,y\rangle$ denotes the standard positive-definite inner product on $\RR^n$. 

Let $S=\{v_1,\ldots,v_n\}$ be a good or cyclic subset with associated string $(a_1,\ldots,a_n)$. Following \cite{lisca-3braids}, we call $$W=\sum_{i=1}^n v_i$$ the \textit{Wu element} of $S$. 
Following \cite{lisca-rational}, we define the integer $I(S)$ to be
$$I(S):=\displaystyle\sum_{i=1}^n(a_i -3).$$

\begin{remark} It is easy to check that if $S$ is a cyclic subset with associated string $\textbf{a}$ and $S^*$ is a cyclic subset whose associated string $\textbf{d}$ is the cyclic dual of $\textbf{a}$, then $I(S)+I(S^*)=0$. 
    \label{rem:Isum}
\end{remark}

\begin{thm} Let $S=\{v_1,\ldots,v_n\}\subset \ZZ^n$ be a good or cyclic subset with $I(S)>0$ and whose Wu element is of the form $W=\sum_{i=1}^nk_i e_i$, where $k_i$ is odd for all $i$. Then $S$ is not cubiquitious.
    \label{thm:Wodd}
\end{thm}

\begin{proof}
Let $z=\frac{1}{2}W$ and let $C$ be the unit cube with centroid $z$. Then for every vector $y\in C$, $d(y,z)^2=\frac{n}{4}$.
Let $x\in \Lambda$, where $\Lambda$ is the lattice generated by $S$. Then $x=\sum_{i=1}^nx_iv_i$, for some integers $x_i$. We will show that $x\not\in C$ by showing that $d(x,z)^2>\frac{n}{4}$.

Let $a_i=\langle v_i, v_i\rangle$. Then
\begin{align*}
    d(x,z)^2&=\left|\left|\sum_{i=1}^{n}x_iv_i-\sum_{i=1 }^n\frac{1}{2}v_i\right|\right|^2\\
    &=\left|\left|\sum_{i=1}^{n}(x_i-\frac{1}{2})v_i\right|\right|^2\\
    &=\sum_{i=1}^n \frac{(2x_i-1)^2}{4}a_i+\sum_{i=1}^{n}\frac{(2x_i-1)(2x_{i+1}-1)}{2}\langle v_i,v_{i+1}\rangle,
\end{align*}
where it is understood that $n+1=1$. 
We will now prove the result when $S$ is negative cyclic; the proofs of the positive cyclic and good cases are similar. Then $\langle v_i,v_{i+1}\rangle =-1$ for all $1\le i\le n-1$ and $\langle v_n,v_1\rangle=1$. Hence,
\begin{align*}
    d(x,z)^2&=\sum_{i=1}^n \frac{(2x_i-1)^2}{4}a_i+\frac{(2x_1-1)(2x_n-1)}{2}-\sum_{i=1}^{n-1}\frac{(2x_i-1)(2x_{i+1}-1)}{2}\\
    &=\sum_{i=1}^n \frac{(2x_i-1)^2}{4}(a_i-2) +\sum_{i=1}^{n-1}(x_i-x_{i+1})^2 +(x_1+x_n-2)(x_1+x_n)+1.
\end{align*}
Note that $(2x_i-1)^2\ge1$ for all $i$ and $(x_1+x_n-2)(x_1+x_n)\ge-1$. Since $\sum_{i=1}^na_i=3n+I(S)$, it follows that 
\begin{align*}
    d(x,z)^2&\ge \sum_{i=1}^n \frac{a_i-2}{4}= \frac{(\sum_{i=1}^n a_i)-2n}{4}=\frac{n+I(S)}{4}>\frac{n}{4}.
\end{align*}
It follows that $x\not\in C$ and so $S$ is not cubiquitous. 
\end{proof}

\section{Proof of Theorem~\ref{thm:1}}\label{sec:proof1}

Recall that $\mathcal{S}_{1a} = \{ (b_1, \dots, b_k, 2, c_l, \dots, c_1, 2) : k + l \ge3 \}$, where $(b_1, \dots, b_k)$ and $(c_1, \dots, c_l)$ are linear duals. It is straightforward to show that
\[
    \mathcal{S}_{1a}^* =\{  (c_1 + b_1, b_2, \dots, b_{k-1}, b_{k}+ c_{l}, c_{l-1}, \dots, c_2) : k + l \ge 3 \}.
\]
Note that by Lemma 4.2 in \cite{simone2020classification}, for the unique minimal length string $\mathbf{a} = (3,2,2,2,2)\in \mathcal{S}_{1a}$, $Y_{\mathbf{a}}^{1}$ does not bound a $\QQ B^4$. Hence we will restrict to elements of $\mathcal{S}_{1a}$ with length at least six and consequently restrict to strings of $\mathcal{S}_{1a}^*$ with length at least two.

Let $\textbf{a}\in \mathcal{S}_{1a}$ and let $\textbf{d}\in \mathcal{S}_{1a}^*$ be its cyclic dual. 
Following the notation in Section~\ref{sec:sharpdbc}, let $X^t_{\textbf{a}}$ denote the negative-definite 4-manifold bounded by $Y^t_{\textbf{a}}$, where $t$ is odd, shown in~Figure~\ref{fig:P}; recall that $t$ indicates the number of half-twists. Endow $H_2(X_{\textbf{a}}^t)$ with a basis given by the 2-handles of $X_{\textbf{a}}^t$ and let $Q$ denote its intersection form. By the lattice analysis completed in Section 6 of \cite{simone2020classification}, there exists a unique lattice embedding $(H_2(X_{\textbf{a}}^t),Q)\to (\ZZ^n,-I)$ (up to composing with an element of $\Aut\ZZ^n$), where $n=\rk(H_2(X_{\textbf{a}}^t))$. 
Moreover, by Theorem~1.7 in \cite{simone2020classification}, $Y_{\textbf{a}}^{-1}$ bounds a $\QQ B^4$, as does $-Y_{\textbf{a}}^{-1}=Y_{\textbf{d}}^{1}$. Our goal is to show that $Y^{1}_{\textbf{a}}$ does not bound a $\QQ B^4$; in fact, we will show that $Y^t_{\textbf{a}}$ does not bound a $\QQ B^4$ for all odd $t>0$.

We first define an intermediate set of strings that we will find useful. Let
\[
    \mathcal{L} = \{ (b_1, b_2, \dots, b_{k-1}, b_{k}+ c_{l}, c_{l-1}, \dots, c_1)\},
\]
where $(b_1, \dots, b_k)$ and $(c_1, \dots, c_l)$ are linear duals.

\begin{lem}
$(d_1,\ldots,d_n)\in\mathcal{L}$ if and only if $(2,d_1,\ldots,d_{n-1},d_n+1)\in\mathcal{L}$ and $(d_1+1,d_2,\allowbreak\ldots,d_{n},2)\in\mathcal{L}$. 
\label{lem:expansions}
\end{lem}

\begin{proof}
    By definition, $(b_1, \dots, b_{k})$ and $(c_1,\ldots,c_l)$ are linear duals if and only if $(2,b_1,\ldots, b_k)$ and $(c_1+1,c_2,\ldots,c_l)$ are linear duals (or equivalently, $(b_1+1,b_2,\ldots, b_k)$ and $(2,c_1,\ldots,c_l)$ are linear duals). The result follows.
\end{proof}

\begin{lem}
    $X^{1}_{\textbf{d}}$ embeds in $m\overline{\mathbb{CP}^2},$ where $m$ is the length of $\textbf{d}\in \mathcal{S}_{1a}^*$, such that its complement is a $\QQ B^4$. Moreover, the total homology class of $X^{1}_{\textbf{d}}$ (i.e., the sum of the homology classes of the 2-handles of $X^{1}_{\textbf{d}}$) has only odd coefficients in the standard basis of $H_2(m\overline{\mathbb{CP}^2}).$
    \label{lem:embed}
\end{lem}

\begin{proof}

\noindent
It is well-known that if $P$ is a linear plumbing whose associated string lies in $\mathcal{L}$ and has length $n$, then $P$ embeds in $n\overline{\mathbb{CP}^2}$ with a $\QQ B^4$ complement. Indeed, one can show that these plumbings are precisely those that can be ``rationally blown down" (see, e.g., the proof of Lemma~2.2 in \cite{parkrationalblowdown}).
We will show this fact explicitly while keeping track of the homology classes of the base spheres of the linear plumbing.

Consider the class $-2e_1\in H_2(\overline{\mathbb{CP}^2})$, where $e_1$ is the standard generator with $e_1^2=-1$. This class can be represented by a $-4$-sphere in $\overline{\mathbb{CP}^2}$ that intersects the $-1$-sphere representing $e_1$ transversely in two positive points, as shown schematically in Figure~\ref{fig:embeddings}. Note that the string $(4)$ is in $\mathcal{L}$ with $k=l=1$. 
Also note that by Lemma \ref{lem:expansions}, any string $\mathcal{L}$ can be obtained from the string $(4)$ by inductively performing the following operations:
\[ (d_1,\ldots,d_k)\to (2,d_1,\ldots,d_{k-1}, d_k+1),\]
\[ (d_1,\ldots,d_k)\to (d_1+1,d_2\ldots,d_k,2).\]

\noindent By blowing up the right point of intersection between the spheres shown in the left of Figure~\ref{fig:embeddings}, we obtain the configuration of spheres in the middle diagram. If we let $\{e_1,e_2\}$ denote the standard basis of $H_2(2\overline{\mathbb{CP}^2})$, then the $-1$-, $-2$- and $-5$-spheres represent the homology classes $e_2$, $e_1 - e_2$ and $-2e_1 - e_2$, respectively. Hence we have the linear plumbing with weights $(-5,-2)$ embedded in $2\overline{\mathbb{CP}^2}$. Note that $(5,2)\in\mathcal{L}$ and the sum of the homology classes of the 2-handles of the plumbing is $-e_1+0\cdot e_2$, which has a single even coefficient. Next, starting with the middle diagram of Figure~\ref{fig:embeddings}, we can either blow up the bottom intersection point or the top right intersection point. These blowups yield linear plumbings embedded in $3\overline{\mathbb{CP}^2}$ with associated strings $(6,2,2)$ and $(2,5,3)$, respectively, both of which are contained in $\mathcal{L}$; moreover, the sum of the homology classes of the 2-handles can be seen to have precisely one even coefficient.
Continuing inductively in this way via blowups, we always obtain a linear plumbing with associated string $(a_1,\ldots,a_n)\in\mathcal{L}$ embedded in $n\overline{\mathbb{CP}^2}$ whose total homology class has precisely one even coefficient. Moreover, any string in $\mathcal{L}$ can be obtained in this way.

\begin{figure}
    \centering
    \begin{overpic}
    [scale=1]{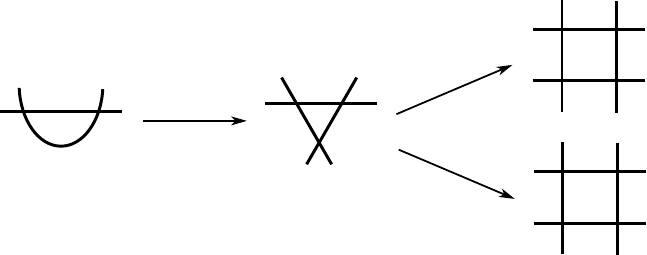}
    \put(7,14){$-4$}
    \put(7,23){$-1$}
    \put(47,25){$-2$}
    \put(42,19){$-5$}
    \put(52,19){$-1$}
    \put(89,36){$-3$}
    \put(82,30){$-5$}
    \put(89,23){$-2$}
    \put(97,30){$-1$}
    \put(89,14){$-2$}
    \put(82,7.5){$-6$}
    \put(89,1){$-1$}
    \put(97,7.5){$-2$}

    \put(18,13.5){\parbox{1in}{\centering\small blow up the\\ right point\\ of intersection}}

    \put(54,11){
    \rotatebox{-22}{\parbox{1in}{\centering\small blow up the\\ bottom point\\ of intersection}}}

    \put(53,26){
    \rotatebox{22}{\parbox{1.1in}{\centering\small blow up the \\ top right point\\ of intersection}}}
    
    \end{overpic}
    \caption{Finding linear plumbings with associated strings in $\mathcal{L}$ embedded in $m\overline{\mathbb{CP}^2}$.}
    \label{fig:embeddings}
\end{figure}
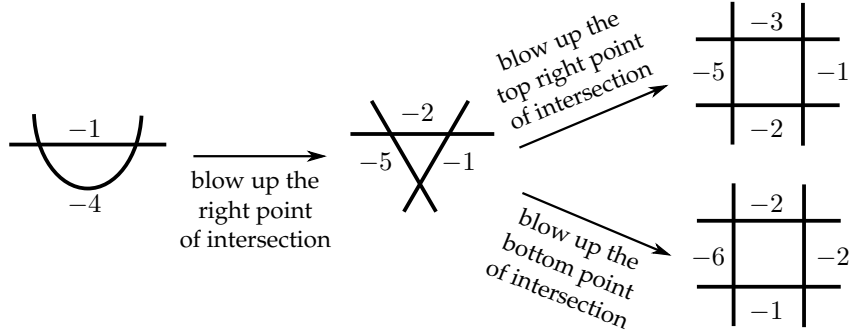

Let
\begin{gather*}
    \textbf{d}=(c_1 + b_1, b_2, \dots, b_{k-1}, b_{k}+ c_{l}, c_{l-1}, \dots, c_2)\in \mathcal{S}_{1a}^*\ \textrm{and} \\
    \textbf{c}=(b_1, b_2, \dots, b_{k-1},\allowbreak b_{k}+ c_{l}, c_{l-1}, \dots, c_1)\in \mathcal{L}.
\end{gather*}
Let $P$ be the linear plumbing embedded in $(k+l-1)\overline{\mathbb{CP}^2}$ with associated string $\textbf{c}$ obtained through the blowup process described above. Let $v_1,\ldots,v_{k+l-1}$ denote the base spheres of $P$ such that $v_1\cdot v_1=-b_1$ and $v_{k+l-1}\cdot v_{k+l-1}=-c_1$. Then either $b_1=2$ or $c_1=2$, but not both. Without loss of generality, assume that $c_1=2$. Then, $v_{k+l-1}=e_{k+l-2}-e_{k+l-1}$ and $v_1=-e_{k+l-2}-e_{k+l-1}+f$ for some vector $f\in H_2((k+l-1)\overline{\mathbb{CP}^2})$, and $v_k\cdot e_{k+l-1}\neq 0$ if and only if $k\in\{1,n\}$. It is easy to see that the unique basis element with even coefficient in the total homology class of $P$ is $e_{k+l-1}$.
A handlebody diagram of a neighborhood of $P$ along with the $-1$-sphere representing $e_{k+l-1}$ is shown in the left side of Figure~\ref{fig:handleslide}.
Orient each unknot counterclockwise and perform a handleslide of the $-c_1$-framed unknot over the $-b_1$-framed unknot using a positively half-twisted band indicated by the green arrow. The attaching circle of the resulting 2-handle will not link the $-1$-framed unknot; moreover, it has framing $-(b_1+c_1)$, and the homology class represented by the sphere given by this 2-handle is $-2e_i+f$. Finally, blow down the $-1$-framed unknot (which removes the homology basis element $e_{k+l-1}$) and ignore the $-b_1$ framed unknot to see the handlebody diagram of $X_{\textbf{d}}^1$ on the right side of Figure~\ref{fig:handleslide} embedded in $(k+l-2)\overline{\mathbb{CP}^2}$. Moreover, since the total homology class of $P$ had exactly one even coefficient, which was the coefficient of $e_{k+l-1}$, it is easy to see that the total homology class of $X_{\textbf{d}}^1$ has all odd coefficients.

Finally, by considering the Mayer--Vietoris sequence applied to $(k+l-2)\overline{\mathbb{CP}^2}=X_{\textbf{d}}^1\cup((k+l-2)\overline{\mathbb{CP}^2}\setminus X_{\textbf{d}}^1)$, it is routine to check that $(k+l-2)\overline{\mathbb{CP}^2}\setminus X_{\textbf{d}}^1$ is a $\QQ B^4$.
\end{proof}

\begin{figure}
    \centering
    \begin{overpic}
    [scale=.7]{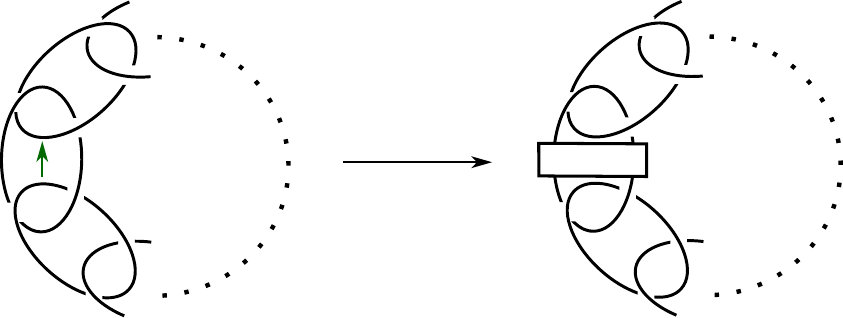}    
    \put(1,33){$-b_1$}
    \put(1,3){$-c_1$}
    \put(12,38.5){$-b_2$}
    \put(12,-3){$-c_2$}
    \put(-5,18){$-1$}

    \put(30,7){\parbox{1.5in}{\centering\small perform a\\ handleslide\\ and blow down;\\ignore the $-b_1$\\framed unknot}}

    \put(50,22){$-(b_1+c_1)$}
    \put(65,32){$-b_2$}
    \put(66,3){$-c_2$}
    \put(69,17.5){$1$}
    \end{overpic}
    \caption{Finding $X_{\textbf{d}}^1$ embedded in $(k+l-2)\overline{\mathbb{CP}^2}$, where $\textbf{d}=(c_1 + b_1, b_2, \dots, b_{k-1}, b_{k}+ c_{l}, c_{l-1}, \dots, c_2)\in \mathcal{S}_{1a}^*$.}
    \label{fig:handleslide}
\end{figure}
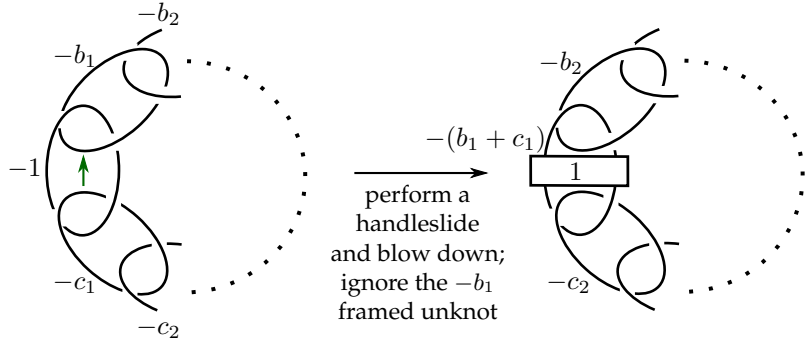

\begin{lem} Let $\textbf{a}\in\mathcal{S}_{1a}$ and let $\textbf{d}\in \mathcal{S}_{1a}^*$ be the cyclic dual of $\textbf{a}.$ For any $t$, $X_{\textbf{a}}^t$ can be turned into $X_{\textbf{d}}^{-t}$ via blowups, blowdowns, and orientation reversal; moreover, this process does not depend on $t$.   
\label{lem:transform}
\end{lem}

\begin{proof} This follows from the proof of Lemma~2.3 in~\cite{simone2020classification}.
\end{proof}

\begin{prop} 
$Y^{t}_{\textbf{a}}$ does not bound a $\QQ B^4$ for all odd $t>0$ and $\textbf{a}\in\mathcal{S}_{1a}$.
\label{prop:noball}
\end{prop}

\begin{proof} 
    Throughout, for any string $\textbf{c}=(c_1,\ldots,c_k)$ and odd integer $t$, we endow $H_2(X^t_{\textbf{c}})$ with the basis given by the 2-handles in the handlebody diagram for $X^t_{\textbf{c}}$ shown in Figure~\ref{fig:P}, ordered according to the order of $\textbf{c}$. Note that the matrices of the intersection forms of $X_{\textbf{c}}^{t_1}$ and $X_{\textbf{c}}^{t_2}$ are identical for all odd integers $t_1$ and $t_2$; this is evident from the handlebody diagrams of $X_{\textbf{c}}^{t_1}$ and $X_{\textbf{c}}^{t_2}$. Hence we denote the intersection form of $X_{\textbf{c}}^t$ by $Q_{\textbf{c}}$ for all odd $t$.
    It follows that there exists a lattice embedding $\psi_{\textbf{c}}^{t_1}:(H_2(X_{\textbf{c}}^{t_1}),Q_{\textbf{c}})\to (\ZZ^k,-I)$ if and only if there exists a lattice embedding $\psi_{\textbf{c}}^{t_2}:(H_2(X_{\textbf{c}}^{t_2}),Q_{\textbf{c}})\to (\ZZ^k,-I)$ and, moreover, these lattice embeddings are identical (up to composing with an element of $\Aut\ZZ^k$). Hence for such lattice embeddings, we will drop the superscript and simply write $\psi_{\textbf{c}}:(H_2(X_{\textbf{c}}),Q_{\textbf{c}})\to (\ZZ^k,-I)$. Finally, we define the subset $C^{\psi}_{\textbf{c}}\subset \ZZ^k$ be the negative cyclic subset whose vectors are the images under $\psi_{\textbf{c}}$ of the basis vectors of $H_2(X_{\textbf{c}})$. Note that $C^{\psi}_{\textbf{c}}$ completely determines  the lattice embedding $\psi_{\textbf{c}}$.
   
    Let $\textbf{a}\in\mathcal{S}_{1a}$ and let $\textbf{d}\in \mathcal{S}_{1a}^*$ be the cyclic dual of $\textbf{a}.$ 
    Let $n$ denote the length of $\textbf{a}$ and let $m$ denote the length of $\textbf{d}$.  
    By Theorem~\ref{thm:s}, $Y^{-1}_{\textbf{a}}$ bounds a $\QQ B^4$, and, moreover, by the lattice analysis undertaken in Section 6 of~\cite{simone2020classification}, there is a unique lattice embedding $\phi_{\textbf{a}}:(H_2(X_{\textbf{a}}),Q_{\textbf{a}})\to (\ZZ^n,-I)$ (up to composing with an element of $\Aut\ZZ^n$).  
    By Lemmas \ref{lem:revorient} and \ref{lem:embed}, $Y^1_{\textbf{d}}=\overline{Y^{-1}_{\textbf{a}}}$ bounds a $\QQ B^4$ and there exists a lattice embedding $\phi_{\textbf{d}}:(H_2(X_{\textbf{d}}),Q_{\textbf{d}})\to (\ZZ^m,-I)$ given by Donaldson's Theorem such that the Wu element of the negative cyclic subset $C^{\phi}_{\textbf{d}}$ has no even coefficients. 
    
    Fix odd $t>0$ and assume that $Y^{t}_{\textbf{a}}$ bounds a $\QQ B^4$, denoted $B$. Then $Y^{-t}_{\textbf{d}}=-Y^t_{\textbf{a}}$ bounds the rational homology ball $-B$. 
    By Donaldson's Theorem, there exist lattice embeddings $\iota_{\textbf{d}}:(H_2(X_{\textbf{d}}),Q_{\textbf{d}})\to (\ZZ^m,-I)$ and $\iota_{\textbf{a}}:(H_2(X_{\textbf{a}}),Q_{\textbf{a}})\to (\ZZ^n,-I)$. 
    By the uniqueness discussed in the previous paragraph, we necessarily have that $C^{\phi}_{\textbf{a}}=C^{\iota}_{\textbf{a}}$. Set $C_{\textbf{a}}:=C^{\phi}_{\textbf{a}}=C^{\iota}_{\textbf{a}}$.
    
    We now show that $C^{\iota}_{\textbf{d}}=C^{\phi}_{\textbf{d}}$.
    By Lemma~\ref{lem:embed}, $X_{\textbf{d}}^1$ embeds in $m\overline{\mathbb{CP}^2}$ with $\QQ B^4$ complement, which we call $B'$. By Lemma~\ref{lem:transform} we can perform blowups and blowdowns in the interior of $X_{\textbf{d}}^1$ embedded in $m\overline{\mathbb{CP}^2}$ along with an ambient orientation reversal to obtain $X_{\textbf{a}}^{-1}$ embedded in $n\overline{\mathbb{CP}^2}$ with complement $-B'$; hence we obtain the unique lattice embedding given by the subset $C_{\textbf{a}}$. We can reverse this process starting with $C_{\textbf{a}}$ and the embedding of $X_{\textbf{a}}^{-1}$ in $n\overline{\mathbb{CP}^2}$ to recover $C^{\phi}_{\textbf{d}}$. Performing the identical procedure to $X_{\textbf{a}}^{t}\cup B$ (cf.~Lemma~\ref{lem:transform}) yields the closed negative-definite 4-manifold $X_{\textbf{d}}^{-t}\cup (-B)$ and changes the negative cyclic subset $C_{\textbf{a}}$ to $C^{\iota}_{\textbf{d}}$. Since we performed the same blowup/blowdown/orientation reversal procedure as above, we necessarily have that $C^{\iota}_{\textbf{d}}=C^{\phi}_{\textbf{d}}$.
    
    It follows that the Wu element of $C^{\iota}_{\textbf{d}}$ has no even coefficients. Moreover, it is easy to see that $I(C_\textbf{a})=-4$;  by Remark \ref{rem:Isum}, $I(C^{\iota}_{\textbf{d}})=4$. Thus by Theorem \ref{thm:Wodd}, $C^{\iota}_{\textbf{d}}$ is not cubiquitous. But by Theorem \ref{thm:sharp}, $X^{-t}_{\textbf{d}}$ is sharp and so by Theorem \ref{thm:cbobstruction}, $C^{\iota}_{\textbf{d}}$ must be cubiquitous, which is a contradiction.
\end{proof}

\begin{remark}
    In the proof of Proposition \ref{prop:noball}, a key point was that there is a unique lattice embedding associated to $X^{-1}_{\textbf{a}}$. It turns out that the same is not true for $X^1_{\textbf{d}}$; there are examples of many lattice embeddings associated to a particular string in $\mathcal{S}_{1a}^*$ that do not satisfy the hypothesis of Theorem \ref{thm:Wodd}. For instance, the intersection lattice in the case $\mathbf{d} = (2,3,4,5,2,3,4,5) \in \mathcal{S}_{1a}^*$ admits distinct embeddings with the coordinates of the Wu element given by $(3,1^{[7]})$ and $(2^{[3]}, 1^{[4]}, 0)$, respectively.
\end{remark}

We are now ready to complete the proof of Theorem \ref{thm:1}.

\proof[Proof of Theorem \ref{thm:1}]
It follows from Theorem~\ref{thm:s} that in order to complete the proof of Theorem~\ref{thm:1}, we need to obstruct $Y^1_\mathbf{a}$ (resp., $Y^{-1}_\mathbf{d}$) for $\mathbf{a} \in \mathcal{S}_{1a} \cup \mathcal{O} \setminus \{ \mathbf{3}_6 \} $ (resp., $\mathbf{d} \in \mathcal{S}_{1a}^* \cup \mathcal{O} \setminus \{ \mathbf{3}_6 \}$) from bounding a $\QQ B^4$. In fact, since for $\mathbf{a} \in \mathcal{S}_{1a} \cup \mathcal{O} \setminus \{ \mathbf{3}_6 \}$ and its cyclic dual $\mathbf{d} \in \mathcal{S}_{1a}^* \cup \mathcal{O} \setminus \{ \mathbf{3}_6 \}$ the manifolds $Y^1_\mathbf{a}$ and $Y^{-1}_\mathbf{d}$ are related by reversing the orientation (Lemma \ref{lem:revorient}), it is sufficient to only prove the first part of the statement. When $\mathbf{a} \in \mathcal{S}_{1a}$, this follows directly from Proposition~\ref{prop:noball}.

Now consider $\mathbf{a} \in \mathcal{O} \setminus \{ \mathbf{3}_6 \}$. Note that $Y^{-1}_\mathbf{a}$ 
bounds a sharp manifold by Theorem~\ref{thm:sharp}. We carry out a computation in the accompanying \textsc{SageMath} notebook available at {\url{https://www.vbrej.xyz/research}} to directly verify that none of the embeddings of the associated lattices are cubiquitous. This implies that $Y^{-1}_\mathbf{a}$ does not bound a $\QQ B^4$, hence the result follows from Theorem~\ref{thm:cbobstruction}.\qedhere

\section{Proof of Theorem~\ref{thm:2}}
\label{sec:ribbons}

To simplify the number of diagrammatic arguments needed, we can use Lemma \ref{lem:mirror}, which states that if $\textbf{a}$ and $\textbf{d}$ are cyclic duals, then the mirror of $B_\textbf{a}^t$ is $B_\textbf{d}^{-t}$. Hence $B_\textbf{a}^t$ is $\chi-$ribbon if and only if $B_\textbf{d}^{-t}$ is $\chi-$ribbon. 

In~\cite{brejevs2020ribbon} it was shown that if $\textbf{a}\in \mathcal{S}_{2}\setminus\mathcal{S}_{2c}$, then $B^0_\textbf{a}$ is $\chi$-ribbon. It follows that $B^0_\textbf{d}$ is also $\chi-$ribbon, where $\textbf{d}$ is the cyclic dual of $\textbf{a}$. Hence if
$\textbf{a}\in(\mathcal{S}_{2}\cup\mathcal{S}_{2}^*)\setminus\mathcal{S}_{2c}$, then $B^0_\textbf{a}$ is $\chi$-ribbon.
The proof in \cite{brejevs2020ribbon} that $B^0_\textbf{a}$ is $\chi$-ribbon for all $\textbf{a}\in \mathcal{S}_{2}\setminus\mathcal{S}_{2c}$ hinges on the observation that if $\mathbf{a}$ contains two substrings that are linear duals of each other, then the corresponding 3-braid closure contains sub-braids $B$ and $C$ that can be cancelled by an isotopy whenever they are connected by a half-twist $\sigma_2 \sigma_1 \sigma_2$; for a careful discussion of this fact, we refer the reader to Section~2 of~\cite{brejevs2020ribbon}. In the following, we enclose $B$ and $C$ in blue and chartreuse rectangles. We exhibit $\chi$-ribbon surfaces for links $B^{-1}_{\textbf{a}}$, where $\textbf{a}\in\mathcal{S}_1$, and for links $B^1_{\textbf{a}}$, where $\textbf{a}\in\mathcal{S}_1\setminus\mathcal{S}_{1a}$, in Figures~\ref{fig:S1ad-1} to~\ref{fig:S1ed+1}. It follows by the remarks above that if $\textbf{a}\in\mathcal{S}_1\cup( \mathcal{S}_1^*\setminus\mathcal{S}_{1a}^*)$, then $B^{-1}_{\textbf{a}}$ is $\chi-$ribbon, and if $\textbf{a}\in(\mathcal{S}_1\setminus\mathcal{S}_{1a})\cup \mathcal{S}_1^*$, then $B^1_{\textbf{a}}$ is $\chi-$ribbon.

It remains to show that if $L=B_{\mathbf{3}_6}^{\pm1}$, then $L$ is not $\chi$-slice. 
Up to isotopy, there are two distinct orientations on $L$: in one case, all strands in the braid whose closure yields $L$ are oriented in the same direction, and in the other, one of the strands is reversed. Routine computations show that in the first case, the signature of $L$ is non-zero, so by Lemma~3.2 in~\cite{DonaldOwens2012}, $L$ does not bound an oriented Euler characteristic one slice surface. In the second case, the Alexander polynomial of $L$ is given by
\[
\Delta_L(t) = (t - 1)^2 \cdot (1 - 6t + 19t^2 - 29t^3 + 19t^4 - 6t^5 + t^6);
\]
the degree six factor is irreducible in $\ZZ[t^{\pm 1}]$, hence $\Delta_L(t)$ does not factor (up to multiplication by units) as $f(t) \cdot f(t^{-1})$ for some $f \in \ZZ[t^{\pm 1}]$. Thus, by Remark~5.4 in~\cite{florens}, $L$ also does not bound an oriented Euler characteristic one slice surface in this case.

Next notice that $L$ is a three-component link and each pair of components forms a Hopf link; hence any two components of $L$ cannot bound a disjoint union of two disks.
If $L=L_1\cup L_2\cup L_3$ is $\chi$-slice, then any Euler characteristic one surface $F$ bounded by $L$ must be non-orientable; hence $F$ must be the disjoint union of a disk and two M\"obius bands. Without loss of generality, assume $L_1$ bounds the disk. But then $L_1\cup L_2$ bounds the disjoint union of a disk and M\"obius band, implying that the Hopf link is $\chi$-slice. This is clearly not possible, however, since the determinant of the Hopf link is not a square.\qedhere

        %\afterpage{\clearpage}
		\begin{figure}[h]
                \vspace{2cm}
			\centering
			\def\svgwidth{0.8\textwidth}
			\makebox[\textwidth][c]{%% Creator: Inkscape 1.0.1 (c497b03c, 2020-09-10), www.inkscape.org
%% Accompanies image file 'qa2_all_2.pdf' (pdf, eps, ps)
\begingroup%
  \makeatletter%
  \providecommand\color[2][]{%
    \errmessage{(Inkscape) Color is used for the text in Inkscape, but the package 'color.sty' is not loaded}%
    \renewcommand\color[2][]{}%
  }%
  \providecommand\transparent[1]{%
    \errmessage{(Inkscape) Transparency is used (non-zero) for the text in Inkscape, but the package 'transparent.sty' is not loaded}%
    \renewcommand\transparent[1]{}%
  }%
  \providecommand\rotatebox[2]{#2}%
  \newcommand*\fsize{\dimexpr\f@size pt\relax}%
  \newcommand*\lineheight[1]{\fontsize{\fsize}{#1\fsize}\selectfont}%
  \ifx\svgwidth\undefined%
    \setlength{\unitlength}{203.25001153bp}%
    \ifx\svgscale\undefined%
      \relax%
    \else%
      \setlength{\unitlength}{\unitlength * \real{\svgscale}}%
    \fi%
  \else%
    \setlength{\unitlength}{\svgwidth}%
  \fi%
  \global\let\svgwidth\undefined%
  \global\let\svgscale\undefined%
  \makeatother%
  \begin{picture}(1,0.54053942)%
    \lineheight{1}%
    \setlength\tabcolsep{0pt}%
    \put(0,0){\includegraphics[width=\unitlength,page=1]{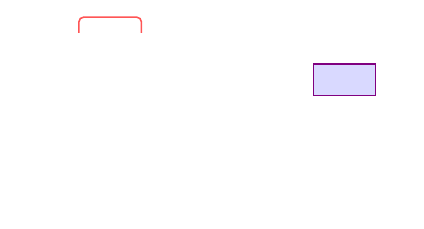}}%
    \put(0.80133909,0.34313908){\color[rgb]{0,0,0}\makebox(0,0)[lt]{\lineheight{1.25}\smash{\begin{tabular}[t]{l}$m$\end{tabular}}}}%
    \put(0,0){\includegraphics[width=\unitlength,page=2]{qa2_all_2.pdf}}%
    \put(0.2303116,0.52184899){\color[rgb]{0,0,0}\makebox(0,0)[lt]{\lineheight{1.25}\smash{\begin{tabular}[t]{l}$\textcolor{red}{-1}$\end{tabular}}}}%
    \put(0.05687029,0.06014595){\color[rgb]{0,0,0}\makebox(0,0)[lt]{\lineheight{1.25}\smash{\begin{tabular}[t]{l}$\xxra{0}{\textrm{expand band + isotopy}}$\end{tabular}}}}%
    \put(0,0){\includegraphics[width=\unitlength,page=3]{qa2_all_2.pdf}}%
    \put(0.56517673,0.38004002){\color[rgb]{0,0,0}\makebox(0,0)[lt]{\lineheight{1.25}\smash{\begin{tabular}[t]{l}$t-1$\end{tabular}}}}%
  \end{picture}%
\endgroup%
}
			\caption{
				Band moves exhibiting $\chi$-ribbon surfaces for closures of braids in family (ii) from Theorem~\ref{thm:murasugi} for $t \geq 1$ and $m \in \ZZ$. In the figure, the cyan rectangle contains $t - 1$ positive full twists $(\sigma_1 \sigma_2)^3$. An analogous positively half-twisted band yields the 2-unlink in the case $t \leq -1$, whilst if $t = 0$ so does an untwisted band between the two twisted strands. Here and further, bands are shown in red and annotated with the number of half-twists in the band with respect to the blackboard framing.
			} \label{fig:fam2}
		\end{figure}
  
	\afterpage{\clearpage}
		\begin{figure}[p]
			\centering
			\def\svgwidth{1.1\textwidth}
			\makebox[\textwidth][c]{%% Creator: Inkscape 1.0.1 (c497b03c, 2020-09-10), www.inkscape.org
%% Accompanies image file 's1ad-1.pdf' (pdf, eps, ps)
\begingroup%
  \makeatletter%
  \providecommand\color[2][]{%
    \errmessage{(Inkscape) Color is used for the text in Inkscape, but the package 'color.sty' is not loaded}%
    \renewcommand\color[2][]{}%
  }%
  \providecommand\transparent[1]{%
    \errmessage{(Inkscape) Transparency is used (non-zero) for the text in Inkscape, but the package 'transparent.sty' is not loaded}%
    \renewcommand\transparent[1]{}%
  }%
  \providecommand\rotatebox[2]{#2}%
  \newcommand*\fsize{\dimexpr\f@size pt\relax}%
  \newcommand*\lineheight[1]{\fontsize{\fsize}{#1\fsize}\selectfont}%
  \ifx\svgwidth\undefined%
    \setlength{\unitlength}{280.09221835bp}%
    \ifx\svgscale\undefined%
      \relax%
    \else%
      \setlength{\unitlength}{\unitlength * \real{\svgscale}}%
    \fi%
  \else%
    \setlength{\unitlength}{\svgwidth}%
  \fi%
  \global\let\svgwidth\undefined%
  \global\let\svgscale\undefined%
  \makeatother%
  \begin{picture}(1,0.82803363)%
    \lineheight{1}%
    \setlength\tabcolsep{0pt}%
    \put(0,0){\includegraphics[width=\unitlength,page=1]{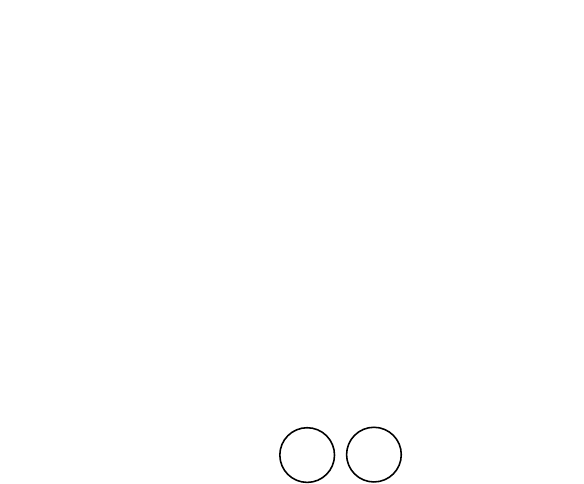}}%
    \put(0.36002004,0.04364495){\color[rgb]{0,0,0}\makebox(0,0)[lt]{\lineheight{1.25}\smash{\begin{tabular}[t]{l}$\xxra{0}[(3)]{\textrm{isotopy}}$\end{tabular}}}}%
    \put(0,0){\includegraphics[width=\unitlength,page=2]{s1ad-1.pdf}}%
    \put(-0.00148179,0.4652067){\color[rgb]{0,0,0}\makebox(0,0)[lt]{\lineheight{1.25}\smash{\begin{tabular}[t]{l}$\xxra{0}[(1)]{\textrm{expand band}}$\end{tabular}}}}%
    \put(0,0){\includegraphics[width=\unitlength,page=3]{s1ad-1.pdf}}%
    \put(0.05123609,0.22564825){\color[rgb]{0,0,0}\makebox(0,0)[lt]{\lineheight{1.25}\smash{\begin{tabular}[t]{l}$\xxra{0}[(2)]{\textrm{cancel duals}}$\end{tabular}}}}%
    \put(0,0){\includegraphics[width=\unitlength,page=4]{s1ad-1.pdf}}%
    \put(0.69102041,0.72751155){\color[rgb]{0,0,0}\makebox(0,0)[lt]{\lineheight{1.25}\smash{\begin{tabular}[t]{l}$\textcolor{red}{-1}$\end{tabular}}}}%
  \end{picture}%
\endgroup%
}
			\caption{
				Band moves for the family $ \mathcal{S}_{1a}^{-1} $. For the definition of blue and chartreuse rectangles in all following figures see Section~2 in~\cite{brejevs2020ribbon}.
			} \label{fig:S1ad-1}
		\end{figure}

	\afterpage{\clearpage}
		\begin{figure}[p]
			\centering
			\def\svgwidth{1.1\textwidth}
			\makebox[\textwidth][c]{%% Creator: Inkscape 1.0.1 (c497b03c, 2020-09-10), www.inkscape.org
%% Accompanies image file 's1bd-1.pdf' (pdf, eps, ps)
\begingroup%
  \makeatletter%
  \providecommand\color[2][]{%
    \errmessage{(Inkscape) Color is used for the text in Inkscape, but the package 'color.sty' is not loaded}%
    \renewcommand\color[2][]{}%
  }%
  \providecommand\transparent[1]{%
    \errmessage{(Inkscape) Transparency is used (non-zero) for the text in Inkscape, but the package 'transparent.sty' is not loaded}%
    \renewcommand\transparent[1]{}%
  }%
  \providecommand\rotatebox[2]{#2}%
  \newcommand*\fsize{\dimexpr\f@size pt\relax}%
  \newcommand*\lineheight[1]{\fontsize{\fsize}{#1\fsize}\selectfont}%
  \ifx\svgwidth\undefined%
    \setlength{\unitlength}{317.62499712bp}%
    \ifx\svgscale\undefined%
      \relax%
    \else%
      \setlength{\unitlength}{\unitlength * \real{\svgscale}}%
    \fi%
  \else%
    \setlength{\unitlength}{\svgwidth}%
  \fi%
  \global\let\svgwidth\undefined%
  \global\let\svgscale\undefined%
  \makeatother%
  \begin{picture}(1,0.92620561)%
    \lineheight{1}%
    \setlength\tabcolsep{0pt}%
    \put(0,0){\includegraphics[width=\unitlength,page=1]{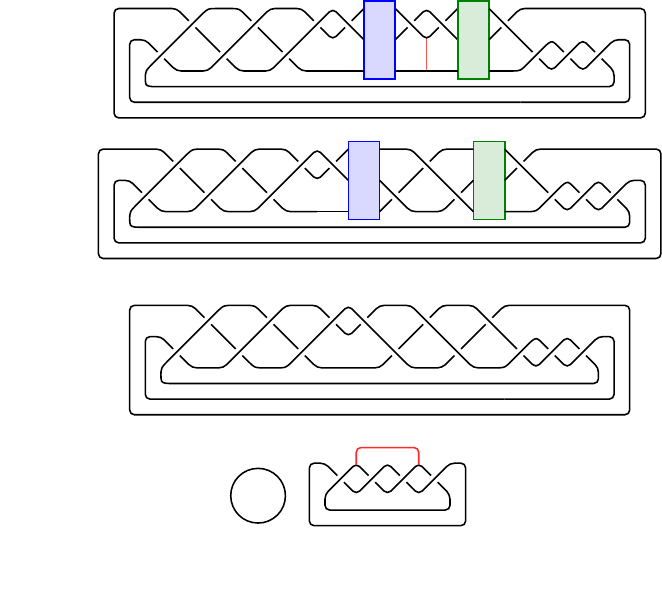}}%
    \put(0.60581081,0.8388963){\color[rgb]{0,0,0}\makebox(0,0)[lt]{\lineheight{1.25}\smash{\begin{tabular}[t]{l}$\textcolor{red}{-1}$\end{tabular}}}}%
    \put(0.56330789,0.26038389){\color[rgb]{0,0,0}\makebox(0,0)[lt]{\lineheight{1.25}\smash{\begin{tabular}[t]{l}$\textcolor{red}{+1}$\end{tabular}}}}%
    \put(-0.00130669,0.60758726){\color[rgb]{0,0,0}\makebox(0,0)[lt]{\lineheight{1.25}\smash{\begin{tabular}[t]{l}$\xxra{0}[(1)]{\textrm{expand band}}$\end{tabular}}}}%
    \put(0.04591881,0.37145982){\color[rgb]{0,0,0}\makebox(0,0)[lt]{\lineheight{1.25}\smash{\begin{tabular}[t]{l}$\xxra{0}[(2)]{\textrm{cancel duals}}$\end{tabular}}}}%
    \put(0.23482083,0.16839031){\color[rgb]{0,0,0}\makebox(0,0)[lt]{\lineheight{1.25}\smash{\begin{tabular}[t]{l}$\xxra{0}[(3)]{\textrm{isotopy}}$\end{tabular}}}}%
    \put(0,0){\includegraphics[width=\unitlength,page=2]{s1bd-1.pdf}}%
    \put(0.17936686,0.03852016){\color[rgb]{0,0,0}\makebox(0,0)[lt]{\lineheight{1.25}\smash{\begin{tabular}[t]{l}$\xxra{0}[(4)]{\textrm{expand band + isotopy}}$\end{tabular}}}}%
  \end{picture}%
\endgroup%
}
			\caption{
				Band moves for the family $ \mathcal{S}_{1b}^{-1} $.
			} \label{fig:S1bd-1}
		\end{figure}

	\afterpage{\clearpage}
		\begin{figure}[p]
			\centering
			\def\svgwidth{1.1\textwidth}
			\makebox[\textwidth][c]{%% Creator: Inkscape 1.0.1 (c497b03c, 2020-09-10), www.inkscape.org
%% Accompanies image file 's1bd+1.pdf' (pdf, eps, ps)
\begingroup%
  \makeatletter%
  \providecommand\color[2][]{%
    \errmessage{(Inkscape) Color is used for the text in Inkscape, but the package 'color.sty' is not loaded}%
    \renewcommand\color[2][]{}%
  }%
  \providecommand\transparent[1]{%
    \errmessage{(Inkscape) Transparency is used (non-zero) for the text in Inkscape, but the package 'transparent.sty' is not loaded}%
    \renewcommand\transparent[1]{}%
  }%
  \providecommand\rotatebox[2]{#2}%
  \newcommand*\fsize{\dimexpr\f@size pt\relax}%
  \newcommand*\lineheight[1]{\fontsize{\fsize}{#1\fsize}\selectfont}%
  \ifx\svgwidth\undefined%
    \setlength{\unitlength}{310.26335709bp}%
    \ifx\svgscale\undefined%
      \relax%
    \else%
      \setlength{\unitlength}{\unitlength * \real{\svgscale}}%
    \fi%
  \else%
    \setlength{\unitlength}{\svgwidth}%
  \fi%
  \global\let\svgwidth\undefined%
  \global\let\svgscale\undefined%
  \makeatother%
  \begin{picture}(1,0.76446683)%
    \lineheight{1}%
    \setlength\tabcolsep{0pt}%
    \put(0,0){\includegraphics[width=\unitlength,page=1]{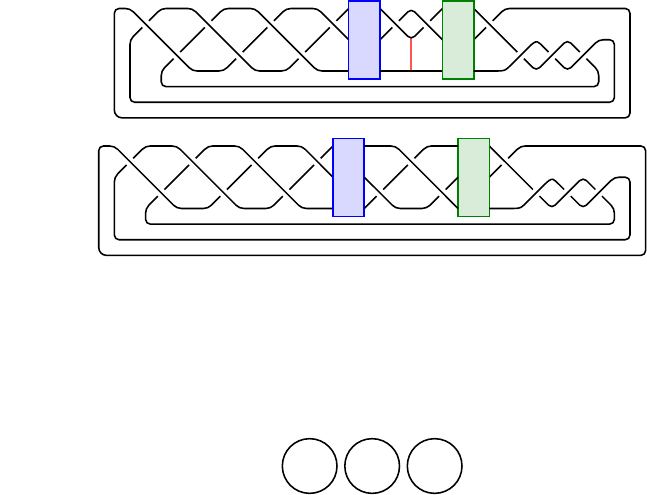}}%
    \put(0.59888736,0.67371987){\color[rgb]{0,0,0}\makebox(0,0)[lt]{\lineheight{1.25}\smash{\begin{tabular}[t]{l}$\textcolor{red}{-1}$\end{tabular}}}}%
    \put(-0.0013377,0.44175712){\color[rgb]{0,0,0}\makebox(0,0)[lt]{\lineheight{1.25}\smash{\begin{tabular}[t]{l}$\xxra{0}[(1)]{\textrm{expand band}}$\end{tabular}}}}%
    \put(0,0){\includegraphics[width=\unitlength,page=2]{s1bd+1.pdf}}%
    \put(0.60855654,0.33529776){\color[rgb]{0,0,0}\makebox(0,0)[lt]{\lineheight{1.25}\smash{\begin{tabular}[t]{l}$\textcolor{red}{-1}$\end{tabular}}}}%
    \put(0.04649911,0.20132124){\color[rgb]{0,0,0}\makebox(0,0)[lt]{\lineheight{1.25}\smash{\begin{tabular}[t]{l}$\xxra{0}[(2)]{\textrm{cancel duals}}$\end{tabular}}}}%
    \put(0.20779573,0.03943408){\color[rgb]{0,0,0}\makebox(0,0)[lt]{\lineheight{1.25}\smash{\begin{tabular}[t]{l}$\xxra{0}[(3)]{\textrm{expand band + isotopy}}$\end{tabular}}}}%
  \end{picture}%
\endgroup%
}
			\caption{
				Band moves for the family $ \mathcal{S}_{1b}^1 $.
			} \label{fig:S1bd+1}
		\end{figure}

	\afterpage{\clearpage}
		\begin{figure}[p]
			\centering
			\def\svgwidth{1.1\textwidth}
			\makebox[\textwidth][c]{%% Creator: Inkscape 1.0.1 (c497b03c, 2020-09-10), www.inkscape.org
%% Accompanies image file 's1cd-1.pdf' (pdf, eps, ps)
\begingroup%
  \makeatletter%
  \providecommand\color[2][]{%
    \errmessage{(Inkscape) Color is used for the text in Inkscape, but the package 'color.sty' is not loaded}%
    \renewcommand\color[2][]{}%
  }%
  \providecommand\transparent[1]{%
    \errmessage{(Inkscape) Transparency is used (non-zero) for the text in Inkscape, but the package 'transparent.sty' is not loaded}%
    \renewcommand\transparent[1]{}%
  }%
  \providecommand\rotatebox[2]{#2}%
  \newcommand*\fsize{\dimexpr\f@size pt\relax}%
  \newcommand*\lineheight[1]{\fontsize{\fsize}{#1\fsize}\selectfont}%
  \ifx\svgwidth\undefined%
    \setlength{\unitlength}{288.4285706bp}%
    \ifx\svgscale\undefined%
      \relax%
    \else%
      \setlength{\unitlength}{\unitlength * \real{\svgscale}}%
    \fi%
  \else%
    \setlength{\unitlength}{\svgwidth}%
  \fi%
  \global\let\svgwidth\undefined%
  \global\let\svgscale\undefined%
  \makeatother%
  \begin{picture}(1,0.96795557)%
    \lineheight{1}%
    \setlength\tabcolsep{0pt}%
    \put(0,0){\includegraphics[width=\unitlength,page=1]{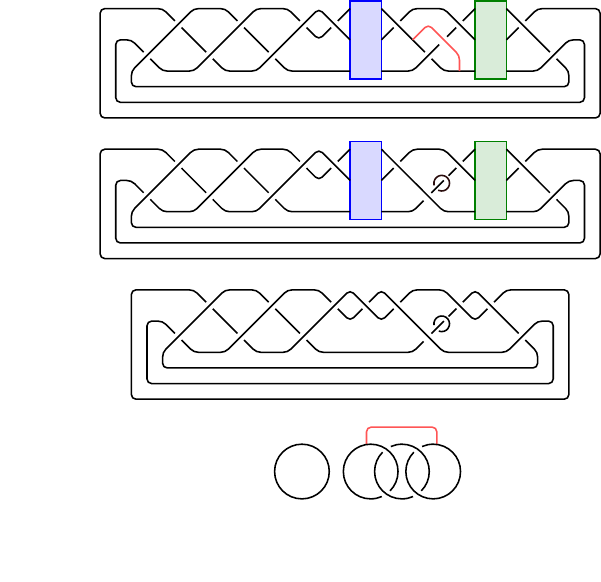}}%
    \put(0.68861555,0.93256023){\color[rgb]{0,0,0}\makebox(0,0)[lt]{\lineheight{1.25}\smash{\begin{tabular}[t]{l}$\textcolor{red}{-1}$\end{tabular}}}}%
    \put(-0.00143897,0.61617251){\color[rgb]{0,0,0}\makebox(0,0)[lt]{\lineheight{1.25}\smash{\begin{tabular}[t]{l}$\xxra{0}[(1)]{\textrm{expand band}}$\end{tabular}}}}%
    \put(0.05056696,0.38214576){\color[rgb]{0,0,0}\makebox(0,0)[lt]{\lineheight{1.25}\smash{\begin{tabular}[t]{l}$\xxra{0}[(2)]{\textrm{cancel duals}}$\end{tabular}}}}%
    \put(0.6418102,0.26837006){\color[rgb]{0,0,0}\makebox(0,0)[lt]{\lineheight{1.25}\smash{\begin{tabular}[t]{l}$\textcolor{red}{-1}$\end{tabular}}}}%
    \put(0.33418511,0.1796778){\color[rgb]{0,0,0}\makebox(0,0)[lt]{\lineheight{1.25}\smash{\begin{tabular}[t]{l}$\xxra{0}[(3)]{\textrm{isotopy}}$\end{tabular}}}}%
    \put(0,0){\includegraphics[width=\unitlength,page=2]{s1cd-1.pdf}}%
    \put(0.21312518,0.04241928){\color[rgb]{0,0,0}\makebox(0,0)[lt]{\lineheight{1.25}\smash{\begin{tabular}[t]{l}$\xxra{0}[(4)]{\textrm{expand band + isotopy}}$\end{tabular}}}}%
  \end{picture}%
\endgroup%
}
			\caption{
				Band moves for the family $ \mathcal{S}_{1c}^{-1} $.
			} \label{fig:S1cd-1}
		\end{figure}

	\afterpage{\clearpage}
		\begin{figure}[p]
			\centering
			\def\svgwidth{1.1\textwidth}
			\makebox[\textwidth][c]{%% Creator: Inkscape 1.0.1 (c497b03c, 2020-09-10), www.inkscape.org
%% Accompanies image file 's1cd+1.pdf' (pdf, eps, ps)
\begingroup%
  \makeatletter%
  \providecommand\color[2][]{%
    \errmessage{(Inkscape) Color is used for the text in Inkscape, but the package 'color.sty' is not loaded}%
    \renewcommand\color[2][]{}%
  }%
  \providecommand\transparent[1]{%
    \errmessage{(Inkscape) Transparency is used (non-zero) for the text in Inkscape, but the package 'transparent.sty' is not loaded}%
    \renewcommand\transparent[1]{}%
  }%
  \providecommand\rotatebox[2]{#2}%
  \newcommand*\fsize{\dimexpr\f@size pt\relax}%
  \newcommand*\lineheight[1]{\fontsize{\fsize}{#1\fsize}\selectfont}%
  \ifx\svgwidth\undefined%
    \setlength{\unitlength}{280.87499423bp}%
    \ifx\svgscale\undefined%
      \relax%
    \else%
      \setlength{\unitlength}{\unitlength * \real{\svgscale}}%
    \fi%
  \else%
    \setlength{\unitlength}{\svgwidth}%
  \fi%
  \global\let\svgwidth\undefined%
  \global\let\svgscale\undefined%
  \makeatother%
  \begin{picture}(1,0.89542782)%
    \lineheight{1}%
    \setlength\tabcolsep{0pt}%
    \put(0,0){\includegraphics[width=\unitlength,page=1]{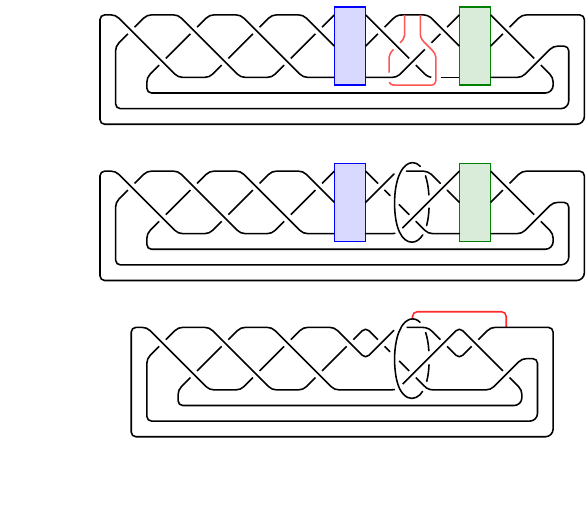}}%
    \put(0.68291173,0.88190283){\color[rgb]{0,0,0}\makebox(0,0)[lt]{\lineheight{1.25}\smash{\begin{tabular}[t]{l}$\textcolor{red}{-2}$\end{tabular}}}}%
    \put(0,0){\includegraphics[width=\unitlength,page=2]{s1cd+1.pdf}}%
    \put(0.76991088,0.37434283){\color[rgb]{0,0,0}\makebox(0,0)[lt]{\lineheight{1.25}\smash{\begin{tabular}[t]{l}$\textcolor{red}{0}$\end{tabular}}}}%
    \put(-0.00147767,0.50283901){\color[rgb]{0,0,0}\makebox(0,0)[lt]{\lineheight{1.25}\smash{\begin{tabular}[t]{l}$\xxra{0}[(1)]{\textrm{expand band}}$\end{tabular}}}}%
    \put(0.05400755,0.23724612){\color[rgb]{0,0,0}\makebox(0,0)[lt]{\lineheight{1.25}\smash{\begin{tabular}[t]{l}$\xxra{0}[(2)]{\textrm{cancel duals}}$\end{tabular}}}}%
    \put(0.17613309,0.03705863){\color[rgb]{0,0,0}\makebox(0,0)[lt]{\lineheight{1.25}\smash{\begin{tabular}[t]{l}$\xxra{0}[(3)]{\textrm{expand band + isotopy}}$\end{tabular}}}}%
  \end{picture}%
\endgroup%
}
			\caption{
				Band moves for the family $ \mathcal{S}_{1c}^1 $.
			} \label{fig:S1cd+1}
		\end{figure}

	\afterpage{\clearpage}
		\begin{figure}[p]
			\centering
			\def\svgwidth{1.1\textwidth}
			\makebox[\textwidth][c]{%% Creator: Inkscape 1.0.1 (c497b03c, 2020-09-10), www.inkscape.org
%% Accompanies image file 's1dd-1.pdf' (pdf, eps, ps)
\begingroup%
  \makeatletter%
  \providecommand\color[2][]{%
    \errmessage{(Inkscape) Color is used for the text in Inkscape, but the package 'color.sty' is not loaded}%
    \renewcommand\color[2][]{}%
  }%
  \providecommand\transparent[1]{%
    \errmessage{(Inkscape) Transparency is used (non-zero) for the text in Inkscape, but the package 'transparent.sty' is not loaded}%
    \renewcommand\transparent[1]{}%
  }%
  \providecommand\rotatebox[2]{#2}%
  \newcommand*\fsize{\dimexpr\f@size pt\relax}%
  \newcommand*\lineheight[1]{\fontsize{\fsize}{#1\fsize}\selectfont}%
  \ifx\svgwidth\undefined%
    \setlength{\unitlength}{333.37495819bp}%
    \ifx\svgscale\undefined%
      \relax%
    \else%
      \setlength{\unitlength}{\unitlength * \real{\svgscale}}%
    \fi%
  \else%
    \setlength{\unitlength}{\svgwidth}%
  \fi%
  \global\let\svgwidth\undefined%
  \global\let\svgscale\undefined%
  \makeatother%
  \begin{picture}(1,0.74429169)%
    \lineheight{1}%
    \setlength\tabcolsep{0pt}%
    \put(0,0){\includegraphics[width=\unitlength,page=1]{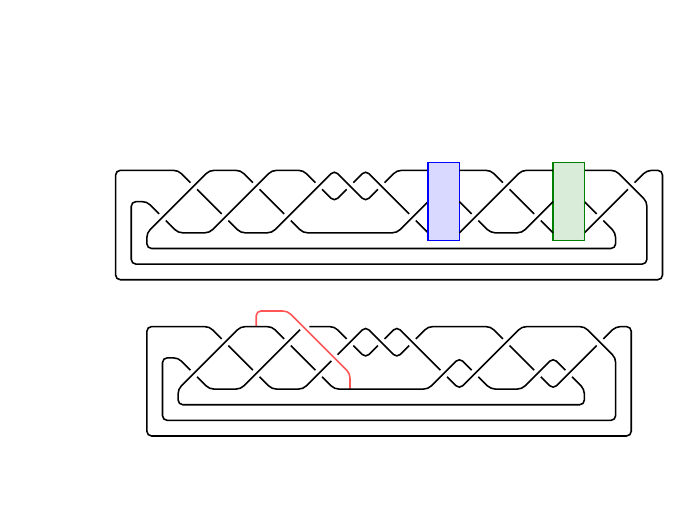}}%
    \put(0.37739136,0.30544994){\color[rgb]{0,0,0}\makebox(0,0)[lt]{\lineheight{1.25}\smash{\begin{tabular}[t]{l}$\textcolor{red}{+1}$\end{tabular}}}}%
    \put(0,0){\includegraphics[width=\unitlength,page=2]{s1dd-1.pdf}}%
    \put(0.66085589,0.73289662){\color[rgb]{0,0,0}\makebox(0,0)[lt]{\lineheight{1.25}\smash{\begin{tabular}[t]{l}$\textcolor{red}{-1}$\end{tabular}}}}%
    \put(0,0){\includegraphics[width=\unitlength,page=3]{s1dd-1.pdf}}%
    \put(-0.00124497,0.40897325){\color[rgb]{0,0,0}\makebox(0,0)[lt]{\lineheight{1.25}\smash{\begin{tabular}[t]{l}$\xxra{0}[(1)]{\textrm{expand band}}$\end{tabular}}}}%
    \put(0.0455024,0.18520606){\color[rgb]{0,0,0}\makebox(0,0)[lt]{\lineheight{1.25}\smash{\begin{tabular}[t]{l}$\xxra{0}[(2)]{\textrm{cancel duals}}$\end{tabular}}}}%
    \put(0.17089279,0.03454198){\color[rgb]{0,0,0}\makebox(0,0)[lt]{\lineheight{1.25}\smash{\begin{tabular}[t]{l}$\xxra{0}[(3)]{\textrm{expand band + isotopy}}$\end{tabular}}}}%
  \end{picture}%
\endgroup%
}
			\caption{
				Band moves for the family $ \mathcal{S}_{1d}^{-1} $.
			} \label{fig:S1dd-1}
		\end{figure}

	\afterpage{\clearpage}
		\begin{figure}[p]
			\centering
			\def\svgwidth{1.1\textwidth}
			\makebox[\textwidth][c]{%% Creator: Inkscape 1.0.1 (c497b03c, 2020-09-10), www.inkscape.org
%% Accompanies image file 's1dd+1.pdf' (pdf, eps, ps)
\begingroup%
  \makeatletter%
  \providecommand\color[2][]{%
    \errmessage{(Inkscape) Color is used for the text in Inkscape, but the package 'color.sty' is not loaded}%
    \renewcommand\color[2][]{}%
  }%
  \providecommand\transparent[1]{%
    \errmessage{(Inkscape) Transparency is used (non-zero) for the text in Inkscape, but the package 'transparent.sty' is not loaded}%
    \renewcommand\transparent[1]{}%
  }%
  \providecommand\rotatebox[2]{#2}%
  \newcommand*\fsize{\dimexpr\f@size pt\relax}%
  \newcommand*\lineheight[1]{\fontsize{\fsize}{#1\fsize}\selectfont}%
  \ifx\svgwidth\undefined%
    \setlength{\unitlength}{318.00376592bp}%
    \ifx\svgscale\undefined%
      \relax%
    \else%
      \setlength{\unitlength}{\unitlength * \real{\svgscale}}%
    \fi%
  \else%
    \setlength{\unitlength}{\svgwidth}%
  \fi%
  \global\let\svgwidth\undefined%
  \global\let\svgscale\undefined%
  \makeatother%
  \begin{picture}(1,0.93586923)%
    \lineheight{1}%
    \setlength\tabcolsep{0pt}%
    \put(0,0){\includegraphics[width=\unitlength,page=1]{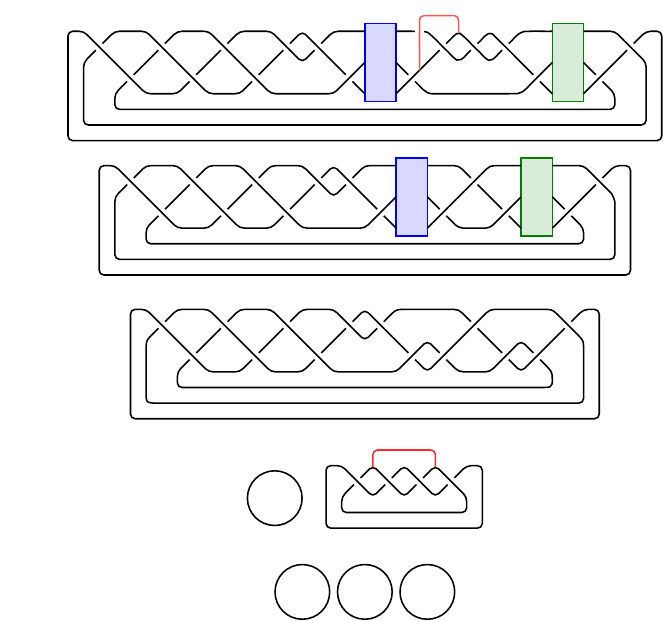}}%
    \put(0.64917974,0.92392337){\color[rgb]{0,0,0}\makebox(0,0)[lt]{\lineheight{1.25}\smash{\begin{tabular}[t]{l}$\textcolor{red}{-1}$\end{tabular}}}}%
    \put(0.58785965,0.26827069){\color[rgb]{0,0,0}\makebox(0,0)[lt]{\lineheight{1.25}\smash{\begin{tabular}[t]{l}$\textcolor{red}{-1}$\end{tabular}}}}%
    \put(-0.00130514,0.59263318){\color[rgb]{0,0,0}\makebox(0,0)[lt]{\lineheight{1.25}\smash{\begin{tabular}[t]{l}$\xxra{0}[(1)]{\textrm{expand band}}$\end{tabular}}}}%
    \put(0.04710623,0.37513068){\color[rgb]{0,0,0}\makebox(0,0)[lt]{\lineheight{1.25}\smash{\begin{tabular}[t]{l}$\xxra{0}[(2)]{\textrm{cancel duals}}$\end{tabular}}}}%
    \put(0.22160552,0.18123562){\color[rgb]{0,0,0}\makebox(0,0)[lt]{\lineheight{1.25}\smash{\begin{tabular}[t]{l}$\xxra{0}[(3)]{\textrm{isotopy}}$\end{tabular}}}}%
    \put(0.17443625,0.03501107){\color[rgb]{0,0,0}\makebox(0,0)[lt]{\lineheight{1.25}\smash{\begin{tabular}[t]{l}$\xxra{0}[(4)]{\textrm{expand band + isotopy}}$\end{tabular}}}}%
  \end{picture}%
\endgroup%
}
			\caption{
				Band moves for the family $ \mathcal{S}_{1d}^1 $.
			} \label{fig:S1dd+1}
		\end{figure}
		
        \begin{sidewaysfigure}
            \centering
            \def\svgwidth{0.95\textwidth}
            %\vspace{2em}
            \makebox[\textwidth][c]{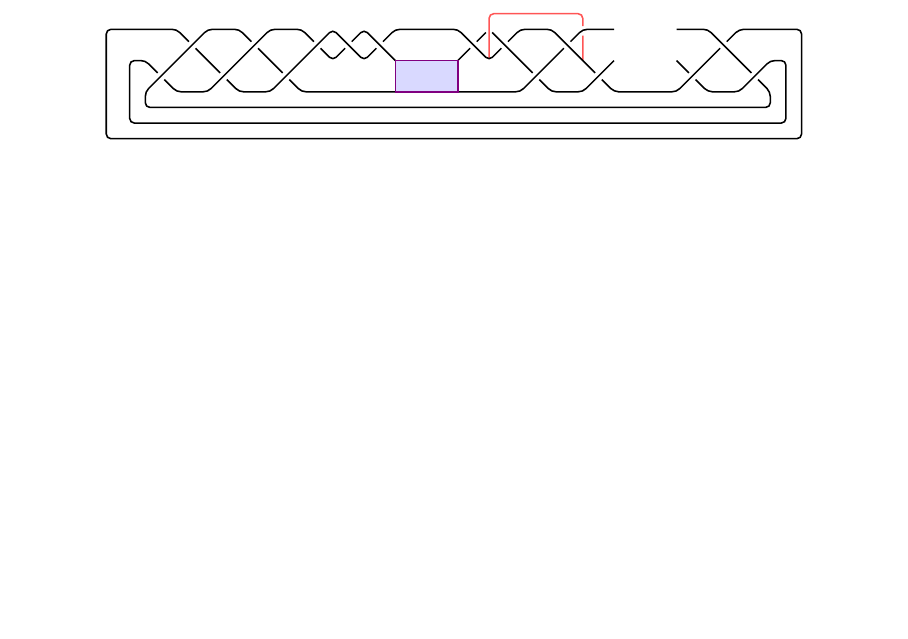}
            \captionsetup{width=.8\linewidth}
		  \vspace*{0.5em}
            \caption{
                Band moves for the family $ \mathcal{S}_{1e}^{-1} $. In step (2) we flype the tangle between the two purple rectangles $x$ times to cancel the crossings contained in the rectangles.
            } \label{fig:S1ed-1}
        \end{sidewaysfigure}

        \begin{sidewaysfigure}
            \centering
            \def\svgwidth{0.9\textwidth}
            %\vspace{2em}
            \makebox[\textwidth][c]{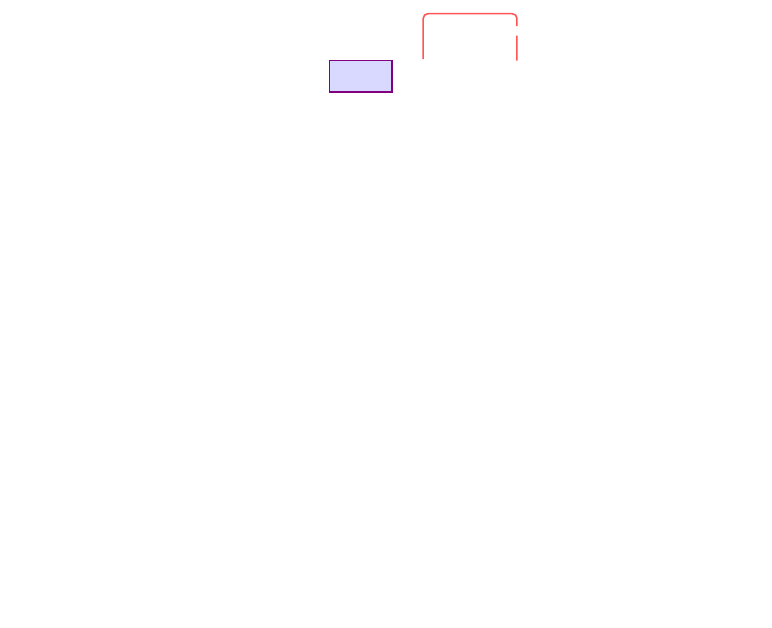}
            \captionsetup{width=.8\linewidth}
		  \vspace*{0.5em}
            \caption{
                Band moves for the family $ \mathcal{S}_{1e}^{1} $. In step (2) we flype the tangle between the two purple rectangles $x$ times to cancel the crossings contained in the rectangles.
            } \label{fig:S1ed+1}
        \end{sidewaysfigure}

\newpage

\begingroup
\renewcommand{\addcontentsline}[3]{}
\printbibliography
\renewcommand{\section}[2]{}% Remove functionality of \section
\endgroup
\end{document}